\documentclass[11pt,english]{article}

\usepackage[margin = 2.5cm]{geometry}

\usepackage{amsthm}
\usepackage{amsmath}
\usepackage{amssymb}
\usepackage{setspace}
\usepackage{mathtools}
\usepackage{graphicx}
\usepackage[hidelinks]{hyperref}
\usepackage{thm-restate}
\usepackage{cleveref}

\usepackage{enumerate}
\usepackage{framed}
\usepackage{subcaption}

\theoremstyle{plain}

\newtheorem*{thm*}{Theorem}
\newtheorem{thm}{Theorem}
\Crefname{thm}{Theorem}{Theorems}

\newtheorem*{lem*}{Lemma}
\newtheorem{lem}[thm]{Lemma}
\Crefname{lem}{Lemma}{Lemmas}

\newtheorem*{claim*}{Claim}
\newtheorem{claim}[thm]{Claim}
\crefname{claim}{Claim}{Claims}
\Crefname{claim}{Claim}{Claims}

\Crefname{prop}{Proposition}{Propositions}

\newtheorem{cor}[thm]{Corollary}
\Crefname{cor}{Corollary}{Corollaries}

\Crefname{conj}{Conjecture}{Conjectures}

\newtheorem{obs}[thm]{Observation}
\Crefname{obs}{Observation}{Observations}

\theoremstyle{definition}
\newtheorem{prob}[thm]{Problem}
\Crefname{prob}{Problem}{Problems}

\newtheorem{defn}[thm]{Definition}
\Crefname{defn}{Definition}{Definitions}

\theoremstyle{remark}

\renewenvironment{proof}[1][]{\begin{trivlist}
\item[\hspace{\labelsep}{\bf\noindent Proof#1.\/}] }{\qed\end{trivlist}}

\newcommand{\eps}{\varepsilon}
\newcommand{\comp}[1]{#1^\mathsf{c}}
\renewcommand{\t}{t}
\newcommand{\good}{good}
\newcommand{\approxAedge}{n ^ {3 / 2} \log n}
\newcommand{\approxBedge}{n ^ {7 / 4} \sqrt{\log n}}
\newcommand{\approxAvx}{n ^ {1 / 2} \log n}
\newcommand{\approxBvx}{n ^ {3 / 4} \sqrt{\log n}}

\newcommand{\degNonTriangular}{\deg_{\text{Non-} \Delta}}
\newcommand{\triangularVs}{triangular}
\newcommand{\myeta}{\eps}
\newcommand{\myK}{\kappa}
 
\newcommand{\sqrtEta}{\sqrt{\myeta}}
\newcommand{\nonNegReals}{\mathbb{R}^{\ge 0}}

\newcommand{\remark}[1]{} 

\newcommand{\remove}[1]{}
\newcommand{\ceil}[1]{
    \lceil #1 \rceil
}
\newcommand{\floor}[1]{
    \lfloor #1 \rfloor
}

\let\oldsqrt\sqrt
\def\sqrt{\mathpalette\DHLhksqrt}
\def\DHLhksqrt#1#2{%
\setbox0=\hbox{$#1\oldsqrt{#2\,}$}\dimen0=\ht0
\advance\dimen0-0.2\ht0
\setbox2=\hbox{\vrule height\ht0 depth -\dimen0}%
{\box0\lower0.4pt\box2}}

\begin{document}

\title{Minimising the number of triangular edges}

\author{
    Vytautas Gruslys\thanks{
        Department of Pure Mathematics and Mathematical Statistics, 	
        University of Cambridge, 
        Wilberforce Road, 
        CB3\;0WB Cambridge, 
        UK;
        e-mail:
        \mbox{\{\texttt{v.gruslys,s.letzter}\}\texttt{@dpmms.cam.ac.uk}}\,.
    }
    \and
    Shoham Letzter\footnotemark[1]
}

\maketitle

\begin{abstract}

    \setlength{\parskip}{\medskipamount}
    \setlength{\parindent}{0pt}
    \noindent

    We consider the problem of minimising the number of edges that are contained
    in triangles, among $n$-vertex graphs with a given number of edges.
    We prove a conjecture of F\"uredi and Maleki that gives an exact formula for
    this minimum, for sufficiently large $n$.

\end{abstract}

\section{Introduction}

    Mantel \cite{mantel} proved that a triangle-free graph on $n$ vertices 
    has at most $\floor{n^2 / 4}$ edges. In other words, a graph on $n$
    vertices with at least $\floor{n^2 / 4} + 1$ edges contains a triangle.  A
    natural question arises from this classical result: how many triangles must
    such a graph have? And, indeed, Rademacher \cite{rademacher} extended
    Mantel's result by showing that any graph on $n$ vertices with $\floor{n^2 /
    4} + 1$ edges contains at least $\floor{n / 2}$ triangles, a bound that can
    readily be seen to be best possible (see \Cref{fig:example}). 
    
    \begin{figure}[h] \centering
        \includegraphics[scale = .8]{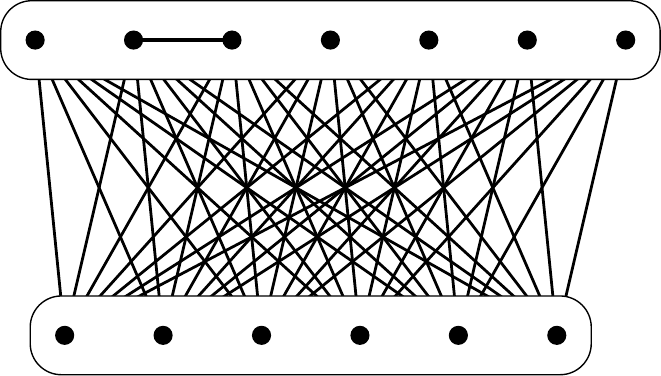}
        \caption{$\floor{n^2/4} + 1$ edges, $\floor{n/2}$ triangles,
        $2\floor{n/2}+1$ triangular edges}
        \label{fig:example}
    \end{figure} 
    
    Erd\H{o}s \cite{erdosHebrew} conjectured that a further generalisation
    holds: any graph on
    $n$ vertices with at least $\floor{n^2 / 4} + l$ edges contains at least $l
    \floor{n / 2}$ triangles, for every $1 \le l < \floor{n / 2}$.
    Erd\H{o}s \cite{erdosHebrew,erdos} proved his conjecture for $l
    \le cn$ for some constant $c > 0$.  It is not hard to see that the bound on
    the number of triangles is best possible, by adding $l$ edges, that do not
    span a triangle, to the larger part of the complete bipartite graph
    $K_{\floor{n / 2}, \ceil{n / 2}}$.  Furthermore, it is not hard to see that
    the bound on $l$ is best possible, by considering a similar construction.

    \begin{figure}[h]\centering
        \includegraphics[scale = .8] {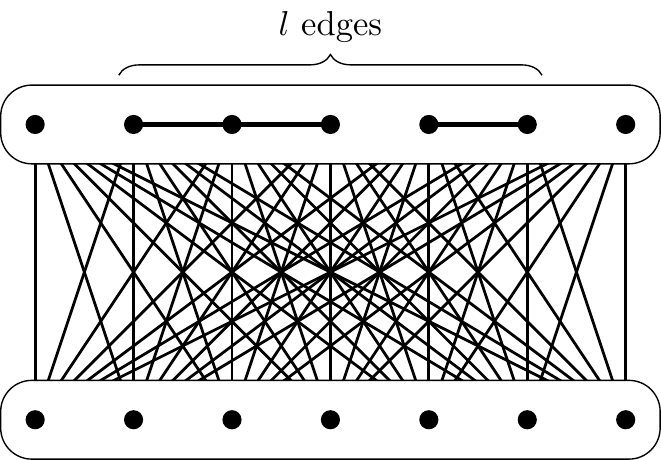}
        \caption{$\floor{n^2/4} + l$ edges, $l \floor{n/2}$ triangles}
        \label{fig:example-l-edges}
    \end{figure}

    Erd\H{o}s's conjecture was resolved by Lov\'asz and Simonovits
    \cite{lovasz-simonovits}, who also characterised \cite{lovasz-simonovitsII}
    the graphs with $n$ vertices and $\floor{n^2 / 4} + l$ edges that minimise
    the number of triangles, for every $l \le c n^2$ where $c > 0$ is fixed.
    Razborov \cite{razborov} determined the asymptotic behaviour of the number
    of triangles in graphs with $n$ vertices and $\floor{n^2 / 4} + l$ edges
    where $l = \Omega(n^2)$.

    In this paper we consider a similar problem, concerning the number of edges
    that are contained in triangles (we shall call such edges \emph{triangular
    edges}) rather than the number of triangles.  The first result in this
    direction was obtained by Erd\H{o}s, Faudree and Rousseau
    \cite{erdos-faudree-rousseau} who proved that any graph with $n$ vertices and
    $\floor{n^2 / 4} + 1$ edges has at least $2 \floor{n / 2} + 1$ triangular
    edges. This bound is best possible (see \Cref{fig:example}).
    
    \remark{is it ok to mention the term `triangular edges' in passing?}

    It is very natural, similarly to the question about the number of triangles,
    to ask how many triangular edges an $n$-vertex graph with $e$ edges must
    have, where $e$ is any integer satisfying $\floor{n^2 / 4} < e \le
    \binom{n}{2}$.  After some thought, a natural example comes to mind. Given
    integers $a, b, c$, we denote by $G(a, b, c)$ the graph with $n = a + b + c$
    vertices, which consists of a clique $A$ of size $a$ and two independent
    sets $B$ and $C$ of sizes $b$ and $c$ respectively, such that all edges
    between $B$ and $A \cup C$ are present, and there are no edges between $A$
    and $C$ (see \Cref{fig:good-graph}). 

    \begin{figure}[h]\centering 
        \includegraphics[scale=.8]{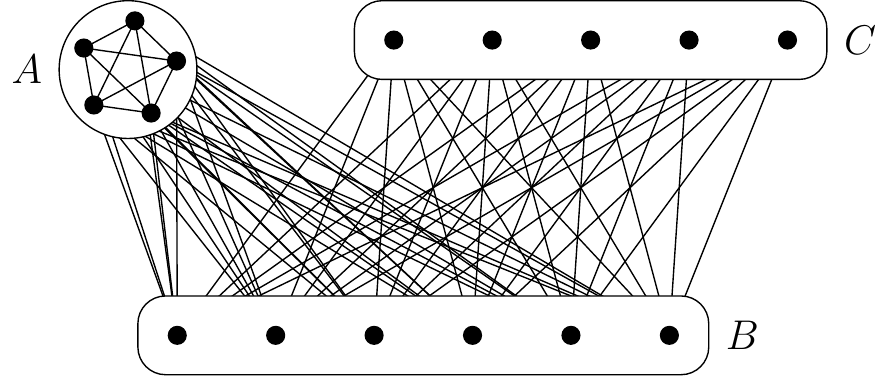}

        \caption{The graph $G(a, b, c)$ (here $a = 4, b = 6, c = 5$)}
        \label{fig:good-graph}
    \end{figure}

    Note that the graph $G(a, b, c)$ has
    $\binom{a}{2} + b(a + c)$ edges, $bc$ of them are non-triangular.
    We remark that the extremal example (depicted in \Cref{fig:example}) for the
    aforementioned result by Erd\H{o}s, Faudree and Rousseau
    \cite{erdos-faudree-rousseau} is isomorphic to $G(2, \floor{n/2}, \ceil{n/2} - 2)$.
    
    \remove{
    We say that a graph is \emph{\good{}} if it is isomorphic to a graph $G(a,
    b, c)$ for some $a, b, c$. 
    }

    F\"uredi and Maleki \cite{furedi-maleki} conjectured that
    the minimisers of the number of triangular edges are graphs of the form
    $G(a, b, c)$, or subgraphs of such graphs.

    \begin{restatable}{conj}{conjMain} 
        \label{conj:chracterise-extremal-examples}

        Let $n$ and $e > \floor{n^2 / 4}$ be integers and let $G$ be a graph
        with $n$ vertices and $e$ edges that minimises the number of triangular
        edges. Then $G$ is a subgraph of a graph $G(a, b, c)$ for some $a, b, c$.
        
    \end{restatable}

    The condition that $G$ is a subgraph of a graph $G(a, b, c)$
    (rather than simply requiring equality) is due to the fact that we specify
    the exact number of edges, so the minimiser may be isomorphic to $G(a, b,
    c)$ with a few edges removed.
    
    \Cref{conj:chracterise-extremal-examples} is a generalisation of the case of
    $n$-vertex graphs with $\floor{n^2 / 4} + 1$ edges, where, as mentioned
    above, the minimiser is indeed $G(2, \floor{n/2}, \ceil{n/2} - 2)$.

    The conjecture implies, in particular, that every
    graph with $n$ vertices and $e$ edges has at least the following number of
    triangular edges:
    \begin{equation*}
        g(n, e) = \min \{ e - bc : a + b + c = n, \binom{a}{2} + b(a + c) \ge e \}.
    \end{equation*}
    F\"uredi and Maleki \cite{furedi-maleki} proved an approximate version of
    the latter statement, which reads as follows. 

    \begin{thm}[F\"uredi and Maleki \cite{furedi-maleki}]
        \label{thm:furedi-maleki}
        Every $n$-vertex graph with $e$ edges has at least $g(n, e) - 3n/2$
        triangular edges.
    \end{thm}

    Our main result is an exact version of \Cref{thm:furedi-maleki}: we shall prove that a
    graph with $n$ vertices and $e$ edges has at least $g(n, e)$ triangular edges,
    provided that $n$ is large enough.

    Before we state our result, we make a few remarks.
    We find it convenient to consider a reformulation of the latter statement.
    Firstly, it turns out to be more convenient to consider the clearly
    equivalent problem of maximising the number of non-triangular edges among 
    $n$-vertex graphs with $e$ edges.  Thus, given a graph $G$, we denote
    by $\t(G)$ the number of \emph{non-triangular} edges in $G$.  Secondly,
    given $n$ and $e$, instead of restricting our attention to $n$-vertex graphs
    with exactly $e$ edges, we consider $n$-vertex graphs with \emph{at least}
    $e$ edges. Since the maximum number of non-triangular edges cannot increase
    by adding edges (assuming that we have at least $\floor{n^2 / 4}$
    edges), this does not affect the problem of maximising the number of
    non-triangular edges, yet it allows us to concentrate on graphs $G(a, b, c)$
    without having to consider their subgraphs. 
    
    We are now ready to state our
    main result.

    \remark{better idea for the notation for the number of non-triangular edges?}

    \begin{thm} \label{thm:main}
        There exists $n_0$ such that for any graph $G$ with at least $n_0$
        vertices there is a graph $H = G(a, b, c)$ (for some integers $a, b, c$)
        that satisfies $|H| = |G|$,
        $e(H) \ge e(G)$ and $\t(H) \ge \t(G)$.  
    \end{thm}

    We note that \Cref{thm:main} comes close to proving
    \Cref{conj:chracterise-extremal-examples} (for sufficiently large $n$) as it
    shows that the minimum number of triangular edges is attained by a
    graph
    $G(a, b, c)$ or a subgraph of $G(a, b, c)$, but we do not explicitly
    show that such graphs are the only minimisers.

    \remark{should we add the comment that our method should (`with high
        confidence') suffice to prove the conjecture for large $n$, with some
    effort which we spare from readers? (I put this remark in the conclusion)}

    \subsection{Structure of the paper}
        
        The proof of our main result, \Cref{thm:main}, is divided into three
        parts, according to the number of edges: graphs that are close to being
        bipartite, i.e., where the number of edges is close to $n^2 / 4$; graphs
        that are close to being complete, i.e., the number of edges is close to
        $\binom{n}{2}$; and the middle range, where the number of edges is
        bounded away from both $n^2 / 4$ and $\binom{n}{2}$.
        
        We state \Cref{thm:almost-bipartite,thm:mid-range,thm:almost-complete},
        which are the theorems corresponding to the aforementioned three ranges,
        in \Cref{sec:overview}, and give an overview of their proofs.
        In \Cref{sec:tools} we describe the tools that
        we shall use and introduce relevant notation.  We prove the three
        theorems in \Cref{sec:almost-bipartite,sec:mid-range,sec:almost-complete}: 
        \Cref{thm:almost-bipartite}, which deals with
        graphs with close to $n^2 / 4$ edges, will be proved in 
        \Cref{sec:almost-bipartite}; the proof of \Cref{thm:mid-range}, for the
        middle range, is the heart of this paper and will be given in
        \Cref{sec:mid-range}; and \Cref{thm:almost-complete} will be proved in
        \Cref{sec:almost-complete}.  We conclude the paper with
        \Cref{sec:conclusion} where we make some remarks and mention open
        problems.

\section{Overview} \label{sec:overview}
    We split the proof of \Cref{thm:main} into three part, according to the
    number of edges $e$. We state the theorems corresponding to these three
    parts here.

    The following theorem deals with $e$ that is close to $n^2 / 4$, i.e.~$e \le
    (1/4 + \delta)n^2$, where $\delta$ is a sufficiently small constant.

    \begin{restatable}{thm}{thmAlmostBipartite} \label{thm:almost-bipartite}

        There exist $n_0$ and $\delta > 0$ such that the following holds. Let
        $G$ be a graph with $n \ge n_0$ vertices and $e$ edges, where $n^2/4 \le e
        \le (1/4 + \delta)n^2$.  Then there is a graph $H = G(a, b, c)$
        that satisfies $|H| = n$, $e(H) \ge e$ and $\t(H) \ge \t(G)$.
        
    \end{restatable}

    The next theorem considers the case where $e$ is bounded away from $n^2/4$
    and $\binom{n}{2}$, namely $(1/4 + \delta)n^2 \le e \le (1/2 - \delta)n^2$
    for any constant $\delta > 0$.

    \begin{restatable}{thm}{thmMidRange} \label{thm:mid-range}

        For every $\delta > 0$ there exist $n_0$ such that the following holds.
        Let $G$ be a graph with $n \ge n_0$ vertices and $e$ edges, where $(1/4 +
        \delta)n^2 \le e \le (1/2 - \delta)n^2$. Then there is a graph $H = G(a, b, c)$
        that satisfies $|H| = n$, $e(H) \ge e$ and $\t(H) \ge \t(G)$.
 
    \end{restatable}

    Finally, we consider the remaining case, when $e$ is close to
    $\binom{n}{2}$, i.e.~$e \ge (1/2 - \delta)n^2$ for $\delta$ sufficiently
    small.

    \begin{restatable}{thm}{thmAlmostComplete} \label{thm:almost-complete}

        There exist $n_0$ and $\delta > 0$ such that the following holds.  Let
        $G$ be a graph with $n \ge n_0$ vertices and $e$ edges, where $e \ge (1/2
        - \delta)n^2$. Then there is a graph $H = G(a, b, c)$
        that satisfies $|H| = n$, $e(H) \ge e$ and $\t(H) \ge \t(G)$.

    \end{restatable}

    We now try to give some insight into our proofs.
    The rough plan in the proof of each of the theorem is the same. Assuming
    that $G$ is a graph on $n$ vertices and at least $e$ edges, that maximises
    the number of non-triangular edges, we first try to obtain information about
    the rough structure of the graph. In each of the cases, we find a partition
    of the vertices $\{A, B, C\}$, where the three parts relate to the three
    parts in a graph $G(a, b, c)$, in a way that will be explained in 
    the proofs.
    In the next stage we use
    lower bounds on the number of non-triangular edges (coming from examples
    $G(a, b, c)$) to estimate the size of the sets $A, B, C$. 
    The final stage uses the estimates on the sizes to conclude that $G$ has the
    required structure, namely, it is isomorphic to the graph $G(|A|, |B|, |C|)$.

    The proofs of the two extremal cases, where $e$ is close to either $n^2/4$
    or $\binom{n}{2}$, are considerably easier than the middle range. The main
    reason for that is that in the extremal cases, it is fairly easy to show that the graph
    $G$ should be close to a graph $G(a, b, c)$, whereas in the middle range,
    getting any handle on the structure of the graph is hard, and the
    initial structural properties that we find are less restrictive than in the two
    extremal cases.
    
    To help us with the proof of the middle range, we introduce two tools. The
    first one is a process of `compression' that allows us to `simplify' a graph
    without decreasing the number of edges or non-triangular edges. The second
    is the `exchange lemma', that allows us to `exchange' edges by
    non-triangular edges (and vice versa), i.e.~by moderately reducing the
    number of edges, we can increase the number of non-triangular edges by a
    given amount. Both these tools will be presented and explained in
    \Cref{sec:tools}.

    \remark{should we call this tool continuity??}

\section{Tools} \label{sec:tools}

    In this section we introduces the tools that will be used in the
    paper.  We start by describing some notation and simple definitions in
    Subsection \ref{subsec:notation}.  We introduce the notion of weighted
    graphs in Subsection \ref{subsec:weighted-graphs} and list some results by
    F\"uredi and Maleki \cite{furedi-maleki} that involve weighted graphs. An
    important tool in the proof of the middle range is the so-called Exchange
    Lemma, \Cref{lem:exchange}. We prove \Cref{lem:exchange} and explain its
    importance in Subsection \ref{subsec:exchange-lemma}.  Our last tool is the
    notion of \emph{compressed graphs} which is a class of graphs with somewhat
    restrictive structure. In Subsection \ref{subsec:compressed-graphs}, we give
    our definition of a compressed graph in and prove \Cref{lem:granulation},
    that shows that it suffices to prove \Cref{thm:main} for compressed graphs.

    \subsection{Notation} \label{subsec:notation}
    
        The following notation is standard. Write $|G|$ for the order of a
        graph $G$ and $e(G)$ for the number of edges in $G$. We denote the
        degree of $u$ in $G$ by $\deg_G(u)$, or $\deg(u)$ if $G$ is clear from
        the context.  Given a set $U$ of vertices of $G$, we denote by $G[U]$
        the graph induced by $G$ on $U$.

        We now turn to notation that is more specific to our context.
        An edge $e \in E(G)$ is called \emph{triangular} if it is contained in a
        triangle of $G$. Similarly, we say that $e$ is \emph{non-triangular} if
        it is not contained in a triangle. We denote by $\t(G)$ the number of
        non-triangular edges of $G$.

        Given a vertex $u$, a vertex $v$ is a \emph{triangular neighbour} of
        $u$, if $uv$ is a triangular edge. Similarly, the \emph{triangular
        neighbourhood} of $u$ is the set of triangular neighbours of $u$, and
        the \emph{triangular degree} of $u$ is the number of triangular edges
        adjacent to $u$.  The notions of a \emph{non-triangular neighbour},
        \emph{non-triangular neighbourhood} and \emph{non-triangular degree} can 
        defined similarly.  We denote the non-triangular degree of $u$ in $G$ by
        $\degNonTriangular(u)$.  A vertex $u$ is called \emph{\triangularVs{}}
        if $\degNonTriangular(u) = 0$, i.e., if all edges adjacent to $u$ are
        triangular. 
        
        We say that a set of vertices $U \subseteq V(G)$ is a \emph{set of
        clones} if any two vertices in $U$ have the same neighbourhood in $G$.
        In particular, a set of clones is independent.  For example, in $G(a, b,
        c)$ the sets $B$ and $C$ are sets of clones. We remark that the notion
        of clones will play an important role in the definition of a compressed
        graph (which will be defined in Subsection
        \ref{subsec:compressed-graphs}).

        We now introduce the natural notion of an optimal graph.

        \begin{defn} \label{def:optimal}
            A graph $G$ on $n$ vertices is called \emph{optimal} if there is
            no graph $H$ on $n$ vertices such that either $\t(H) > \t(G)$ and
            $e(H) \ge e(G)$ or $e(H) > e(G)$ and $\t(H) \ge \t(G)$. 
            
            In other
            words, $G$ is optimal if it maximises $\t(G)$ among graphs with $n$
            vertices and at least $e(G)$ edges and, in addition, it maximises
            $e(G)$ among graphs with $n$ vertices and at least $\t(G)$
            non-triangular edges.         
        \end{defn}
        
        It clearly suffices to prove the main result, \Cref{thm:main}, for
        optimal graphs.
        The following observation is a simple property of optimal graphs. 
        
        \begin{obs} \label{obs:two-vs}
            Let $G$ be an optimal graph and let $u$ and $v$ be vertices in $G$.
            Then either
            $\deg(u) \ge \deg(v) - 1$ or $\degNonTriangular(u) \ge
            \degNonTriangular(v) - 1$.
        \end{obs}

        \begin{proof}
            Suppose that $\deg(u) \le \deg(v) - 2$ and $\degNonTriangular(u) \le
            \degNonTriangular(v) - 2$. Consider the graph $G'$ obtained by removing
            the edges adjacent to $u$ and adding the edges between $u$ and the
            neighbours of $v$. Then $e(G') \ge e(G) - \deg(u) + \deg(v) - 1 > e(G)$
            and, similarly, $\t(G') > \t(G)$, contradicting the assumption that
            $G$ is optimal.
        \end{proof}

        We shall use big-O notation extensively throughout this paper, so, for
        the sake of clarity, we briefly explain how we interpret the symbols $O$,
        $o$ and $\Omega$.
        We write $f(n) = O(g(n))$ if there exists an absolute constant $C > 0$
        such that $|f(n)| \le C|g(n)|$. In particular, the expression $f(n) =
        g(n) + O(h(n))$ consists of the following inequalities: $g(n) - C|h(n)| \le f(n)
        \le g(n) + C|h(n)|$. Similarly, $f(n) = o(g(n))$ means that
        $\lim_{n \rightarrow \infty} |f(n) / g(n)| = 0$. Finally, we write $f(n)
        = \Omega(g(n))$ if $f(n) \ge C \cdot g(n)$ for an absolute constant $C > 0$.

        Throughout this paper, we omit integer parts whenever they do not affect
        the argument.
        We always assume that $n$ is large.

    \subsection{Weighted graphs} \label{subsec:weighted-graphs}

        \remark{I added some definitions here: the notion of a `weighted
            subgraphs' which is useful for a few statements, and the
        generalisation of degree to weighted graphs}

        Our most basic tool is the concept of a \emph{weighted graph}, which is
        a graph whose vertices have non-negative weights. The \emph{total
        weight} of a weighted graph $G$ is the sum of the weights of its
        vertices and is denoted by $|G|$. Throughout this paper we require that the
        number of vertices of a weighted graph does not exceed its total weight.
        Equivalently, we require that the average weight of a vertex in a weighted
        graph is at least $1$.

        Given weighted graphs $G$ and $H$ we say that $H$ is a \emph{weighted
        subgraph} of $G$ if, as graphs, $H$ is a subgraph of $G$. Note that this
        definition does not impose any conditions on the weight function of $H$. In
        particular, if $H$ is a weighted subgraph of $G$ then the weight in $H$ of a
        vertex in $H$ may be larger than its weight in $G$.

        Given a weighted graph $G$ with weight function $w : V(G) \to \nonNegReals$
        we define $e(G)$ to be the sum of $w(u)w(v)$ over all edges $uv$ of $G$.
        Similarly, we define $\t(G)$ to be the same sum over the
        \emph{non-triangular} edges of $G$. Note that any graph $G$ can be seen
        either as a graph or as a weighted graph whose every vertex has weight $1$,
        and the definitions of $|G|, e(G)$ and $\t(G)$ do not depend on the point of
        view. 
        
        The notions of the degree and the non-triangular degree of a
        vertex may be similarly generalised to weighted graphs. For instance, the
        degree of $u$ in a weighted graph $G$ is the sum
        of weights of the neighbours of $u$.
        We use the notation $\deg(u)$ and
        $\degNonTriangular(u)$ for the degree and the non-triangular degree of a
        vertex $u$ in a weighted graph.
        \remark{do we need to include $G$ in the notation here? Do we ever use
        this?}

        We now define a \emph{\good{} weighted graph} (see also
        \Cref{fig:good-weighted}). 

        \begin{defn} \label{def:good}
            We call a weighted graph $G$ \good{} if its vertex set can be
            partitioned into a set $K$ that induces a clique, and a pair $(u,
            v)$ of adjacent vertices, such that $uv$ is the only non-triangular
            edge in $G$.
        \end{defn}

        \begin{figure}[h]\centering
            \includegraphics[scale=1]{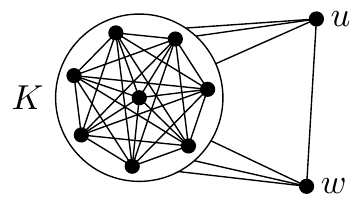}
            \caption{a \good{} weighted graph}
            \label{fig:good-weighted}
        \end{figure} 

        A \good{} weighted graph can be viewed as a
        weighted analogue of a graph $G(a, b, c)$ (see \Cref{fig:good-graph}).
        Indeed, a graph $G(a, b, c)$ can be represented by a \good{} weighted
        graph, by replacing the
        independent sets (of sizes $b$ and $c$) by vertices of weight $b$ and
        $c$ respectively. However, a \good{} weighted graph may have non-integer
        weights.
        We remark that we shall often use this correspondence
        between an independent set of clones $I$  and a vertex of weight $|I|$ with
        the same neighbourhood.
        \remark{again, can the good notion be avoided?}

        Motzkin and Straus \cite{motzkin-straus} used weighted graphs to give an
        alternative proof of Tur\'an's theorem \cite{turan}. They observed that
        Tur\'an's theorem for weighted graphs is very easy: given a weighted
        graph $G$, there exists a weighted graph $H$ that satisfies $|H| = |G|$
        and $e(H) \ge e(G)$, and, as a graph, is a complete subgraph of $G$.
        Therefore, among $K_{r+1}$-free weighted graphs with total weight
        $\alpha \ge r$, $e(G)$ is maximised when $G$ is a complete graph with
        $r$ vertices whose every vertex has weight $\alpha / r$.  If $\alpha /
        r$ is an integer, then this corresponds to a complete $r$-partite graph,
        implying Tur\'an's theorem. However, if $\alpha / r$ is not an integer,
        then this argument gives only an approximate form of Tur\'an's theorem,
        and Motzkin and Straus needed an additional argument to recover the full
        theorem.

        F\"uredi and Maleki \cite{furedi-maleki} modified this result to also
        give $t(H) \ge t(G)$ at the cost of making the structure of $H$ more
        complicated (here we use \Cref{def:good} of a \good{} weighted graph,
        see \Cref{fig:good-weighted}).
        
        \begin{lem}[F\"uredi and Maleki \cite{furedi-maleki}]
        \label{lem:approx-by-good-weighted-graph}
            Let $G$ be a weighted graph with $\t(G) > 0$. Then there exists a
            \good{} weighted subgraph $H$ of $G$ that satisfies $|H| = |G|$, $e(H)
            \ge e(G)$ and $\t(H) \ge \t(G)$.
        \end{lem}

        \remark{we used to have a longer description of the structure of $H$, but it
        was just the description of a good weighted graph. Also, I used the notion
        of `weighted subgraph' which was defined at the beginning of the subsection}

        We will use both this result and the key observation that leads to its
        proof. We state and prove this observation next, but we do not present
        the careful analysis that the aforementioned authors perform to complete
        the proof of \Cref{lem:approx-by-good-weighted-graph}.

        \begin{lem}[F\"uredi and Maleki \cite{furedi-maleki}]
        \label{lem:process-no-empty-triangle}
            Let $G$ be a weighted graph and suppose that $I$ is an independent set
            of three vertices. Then there exists a weighted graph $H$, which can
            be obtained from $G$ by removing one of the vertices in $I$ and changing
            the weights of the other two vertices in $I$, such that $|H| = |G|$,
            $e(H) \ge e(G)$ and $\t(H) \ge \t(G)$.
        \end{lem}

        \begin{proof}
            Denote $I = \{u_1, u_2, u_3\}$, $d_i = \deg(u_i)$ and $t_i =
            \degNonTriangular(u_i)$. It is not hard to see that there exist $s_1,
            s_2, s_3$, not all $0$, such that $s_1 d_1 + s_2 d_2 + s_3 d_3 \ge 0$,
            $s_1 t_1 + s_2 t_2 + s_3 t_3 \ge 0$ and $s_1 + s_2 + s_3 = 0$. Denote by
            $G_\lambda$ the weighted graph obtained by adding $\lambda s_i$ to the
            weight $w(u_i)$ of $u_i$ for each $i \in [3]$; this definition is valid for the
            values of $\lambda$ for which $w(u_i) + \lambda s_i \ge 0$ for $i
            \in [3]$. Pick 
            $\lambda > 0$ such that $w(u_i) + \lambda s_i \ge 0$ 
            for $i \in [3]$ with equality for at least one value, say $1$.
            Then
            $|G_\lambda| = |G|$, $e(G_\lambda) \ge e(G)$ and $t(G_\lambda) \ge
            e(G)$, so the weighted graph $H = G_\lambda \setminus \{u_1\}$ satisfies
            the requirements of the lemma.
        \end{proof}

        F\"uredi and Maleki deduce their main result from
        \Cref{lem:approx-by-good-weighted-graph}. We present their theorem with
        minor modifications, which make it more suitable for our application.

        \begin{cor}[F\"uredi and Maleki \cite{furedi-maleki}]
        \label{cor:approx-by-integer-example}
            Let $G$ be a weighted graph $G$ with $|G| = n$.  Then there exists a
            graph $H = G(a, b, c)$ satisfying $|H| = n$, $e(H) \ge e(G)$
            and $\t(H) \ge \t(G) - 5n$.
        \end{cor}

        \begin{proof} Let $G$ be a weighted graph with total weight $n$. We may
            assume that $\t(G) > 5n$ because otherwise the complete graph $K_n$
            satisfies the requirements (here we use the fact that, according to
            our definition of a weighted graph, $G$ has at most $n$ vertices, so
            $e(G) \le \binom{n}{2}$).  Let $H$ be the \good{} weighted graph
            that satisfies $|H| = n$, $e(H) \ge e(G)$ and $\t(H) \ge \t(G)$,
            whose existence is ensured by
            \Cref{lem:approx-by-good-weighted-graph}.  Since $H$ is good, there
            exists a partition $\{K, \{u, v\}\}$ of $V(H)$ such that $K$
            induces a clique and $uv$ is the only non-triangular edge. Denote
            the sum of weights (in $H$) of the vertices in $K$ by $\alpha$ and
            the weights of $u$ and $v$ by $\beta$ and $\gamma$; we may assume
            that $\beta \ge \gamma$. Then $e(G) \le \alpha ^ 2 / 2 + \alpha
            \beta + \beta \gamma$ and $\t(G) = \beta \gamma$. 

            We now show that for some integers $a, b, c \ge 0$ the graph $G' =
            G(a,b,c)$ has the desired properties. It is enough to choose $a,b,c$ so
            that
            \begin{align}
                a+b+c & \, = \, n,
                    \label{eq:order} \\
                \binom{a}{2} + (n-b)b \, \ge \, &\frac{\alpha^2}{2} + (n - \beta)\beta,
                    \label{eq:ineq-for-eG}\\
                bc \, \ge \, \beta \gamma &- 5n.
                    \label{eq:ineq-for-tG}
            \end{align}

            Of course, the plan is to set $a \approx \alpha, b \approx \beta, c
            \approx \gamma$, but there are some tedious details to check. We set $a
            = \ceil{\alpha} + 2$ and, depending on whether $\beta \ge n/2$ or
            $\beta < n/2$, either $b = \floor{\beta}$ or $b = \ceil{\beta}$.
            Finally, we set $c = n - a - b$. Note that by the assumption that $\t(G)
            \ge 5n$, it follows that $\beta, \gamma > 5$. In particular,
            since $c \ge \gamma - 4$, $c$ is positive. Now, (\ref{eq:order}) is
            immediate from the definition; (\ref{eq:ineq-for-tG}) is immediate from the
            fact that $b \ge \beta - 1 > 0$ and $c \ge \gamma - 4 > 0$; and
            the only case when (\ref{eq:ineq-for-eG}) is not immediate is when
            $(n-1)/2 \le \beta \le (n+1)/2$. However, in this case the difference
            between $(n - \beta)\beta$ and $(n - b)b$ is at most $1$, and it is
            compensated by the difference between $\binom{a}{2}$ and $\alpha^2/2$.
        \end{proof}

    \subsection{Exchange lemma} \label{subsec:exchange-lemma}
        
        The following lemma, \Cref{lem:exchange}, will prove very useful in the
        proof of our main result in the middle range. Roughly speaking, it says that
        there is a positive number $\zeta$, which we informally call the
        `exchange rate', with the following property. For any graph $G$, not too
        dense and not too sparse, and any number $x$, not too big and not too
        small, we can exchange $x$ edges of $G$ for at least $\zeta x$
        non-triangular edges.  That is, there is a graph $H$ such that $|H| =
        |G|$, $e(H) \ge e(G) - x$ and $\t(H) \ge \t(G) + \zeta x$. Similarly, we
        can exchange $x$ non-triangular edges for at least $\zeta x$ edges.

        This tool is very useful for us, because now we can arrive at a
        contradiction by finding a graph $G$ whose either parameter $e(G)$ or
        $\t(G)$ is too large, even if the other parameter is slightly smaller
        than needed.  

        For any positive integer $n$ and real $e \le \binom{n}{2}$, denote by
        $\t(n, e)$ the maximum number of non-triangular edges 
        among $n$-vertex graphs with at least $e$ edges. Note that if $e \le
        \lfloor n^2 / 4 \rfloor$, then $\t(n, e) = \lfloor n^2 / 4 \rfloor$.
        Moreover, for any $n$, the function $\t(n, e)$ is a non-increasing
        function of $e$.

        \begin{lem} \label{lem:exchange}
            For any $\delta > 0$ there exist $\zeta, \eps, C > 0$ and $n_0$ such
            that the following holds. Here $G$ is a weighted graph with $n \ge n_0$
            vertices and $x$ is a real satisfying $C n \le
            x \le \eps n^2$.
            \begin{enumerate}
                \item
                    If $e(G) \ge e + x$ for some real $e$ satisfying $ n^2 / 4
                    \le e \le (1/2 - \delta)
                    n^2$, then $\t(G) \le \t(n, e) - \zeta x$.
                \item
                    If $\t(G) \ge \t(n, e) + x$ for some real $e \ge (1/4 +
                    \delta) n^2$, then $e(G) \le e - \zeta x$.
            \end{enumerate}
        \end{lem}

        Before turning to the proof of \cref{lem:exchange}, we give a brief
        overview of the proof.  To prove the first statement, we note that by
        \Cref{lem:approx-by-good-weighted-graph}, we may assume that $G$ is
        \good. We shift the weights of the vertices in $G$ so as to
        increase $\t(G)$ while decreasing $e(G)$ only slightly. An upper bound
        on $\t(G)$ then follows from \Cref{cor:approx-by-integer-example}. The
        second statement is proved in a similar way.
        
        \remark{does the overview make sense?}

        \begin{proof} [ of \Cref{lem:exchange}]
            Let $\delta \in (0, 1/10)$. To prove the first statement, suppose
            that $n, e, x$ satisfy $e \le (1 / 2 - \delta) n ^ 2$ and $C n \le x
            \le \eps n ^ 2$ for constants $C$ and $\eps$ that will be
            determined later.  Let $G$ be a weighted graph such that $|G| = n$
            and $e(G) \ge e + x$. We note that $\t(n, e) \ge \frac{\delta ^ {3 /
            2}}{2} n ^ 2$.  Indeed, the graph $G(a, b, c)$ where $c =
            \frac{\delta}{2} n$, $b = \sqrt{\delta} n$ and $a = n - b - c$ has
            at least $(1 / 2 - \delta)n ^ 2$ edges and $\frac{\delta ^ {3 /
            2}}{2} n ^ 2$ non-triangular edges. By taking $\eps, \zeta$ to
            satisfy $\eps \zeta \le \frac{\delta ^ {3 / 2}}{4}$, we may assume
            that
            \begin{equation} \label{eqn:exchange-lower-bd-t}
                \t(G) \ge \frac{\delta ^ {3 / 2}}{4} n ^ 2,
            \end{equation}
            because otherwise we get $\t(G) \le \t(n, e) - \zeta x$ for free.
            By \Cref{lem:approx-by-good-weighted-graph}, we may assume that
            $G$ is a \good{} weighted graph, so $V(G)$ can be partitioned
            into a clique $K$ and two adjacent vertices $u$ and $v$ such that
            $uv$ is the only non-triangular edge. Denote by $\alpha$ the sum of the
            weights of the vertices in $K$ and let $\beta$ and $\gamma$ be the
            weights of $u$ and $v$ respectively.  By Inequality
            (\ref{eqn:exchange-lower-bd-t}), we have $\beta, \gamma \ge
            \frac{\delta ^ {3 / 2}}{4} n$. Moreover, the removal of the edges
            spanned by $K$ would make $G$ bipartite, so we have $e(G) \le n^2 /
            4 + \alpha^2/2 \le n^2/4 + \alpha n/2$. Recall that $e(G) \ge e + x
            \ge n^2/4 + x$, hence $\alpha \ge 2x / n$.
            
            Let $G'$ be a weighted graph obtained by increasing the weight of
            $u$ by $x / n$ and decreasing the weights of the vertices in $K$ so
            that their new sum of weights is $\alpha - x / n$. It is easy to check that
            $e(G') \ge e(G) - x \ge e$ and $\t(G') = (\beta + x / n) \gamma \ge
            \t(G) + \frac{\delta ^ {3 / 2}}{4} x$.  Furthermore, it follows from
            \Cref{cor:approx-by-integer-example} that $\t(G') \le \t(n, e) +
            5n$, hence $\t(G) \le \t(n, e) + 5n - \frac{\delta^{3 / 2}}{4} x
            \le \t(n, e) - \left( \frac{\delta^{3 / 2}}{4} - \frac{5}{C} \right)
            x$.  By taking $C$ large and $\zeta$ small with respect to $\delta$,
            we can ensure that $\t(G) \le \t(n, e) - \zeta x$.

            To prove the second statement, suppose that $n, e, x$ satisfy $e \ge (1
            / 4 + \delta)n ^ 2$ and $10n < x \le \eps n ^ 2$ for a sufficiently
            small constant $\eps > 0$. Let $G$ be a weighted graph such that
            $|G| = n$ and $\t(G) \ge \t(n, e) + x$.  Note that by taking $\eps,
            \zeta$ to satisfy $\eps \zeta \le \delta / 2$, we may assume that 
            \begin{equation} \label{eqn:exchange-lower-bd-e}
                e(G) \ge \left( \frac{1}{4} + \frac{\delta}{2} \right) n ^ 2,
            \end{equation}
            because otherwise we can conclude immediately that $e(G) \le e -
            \zeta x$.
            Furthermore, by \Cref{lem:approx-by-good-weighted-graph} we may
            assume $G$ is a \good{} weighted graph. Denote by $K$, $u$ and
            $v$ the corresponding clique and vertices and let $\alpha$
            be the total weight of $K$ and let $\beta$ and $\gamma$ be the
            weights of $u$ and $v$. As before, it follows from
            Inequality~(\ref{eqn:exchange-lower-bd-e}) that $\alpha \ge
            \sqrt{\delta} n$. Moreover, $K$ must contain at least two vertices,
            so in particular a vertex $w \in K$ whose weight does not exceed
            $\alpha / 2$. Let $G'$ be the weighted graph obtained by reducing
            the weight of $v$ by $x / 2n$ (note that $\beta \gamma = \t(G) \ge
            x$, so $\gamma \ge x / n$) and increasing the weight of $w$ by the
            same amount . Then, since $x > 10n$, 
            \begin{equation} \label{eqn:exchange-lower-bd-t-G'}
                \t(G') = \beta (\gamma - x / 2n) \ge \t(G) - x / 2 \ge \t(n, e) +
                x / 2 > \t(n, e) + 5n.
            \end{equation}
            Furthermore, since $\alpha \ge \sqrt{\delta}n$, $$e(G') \ge e(G) +
            \frac{x}{2n} \cdot \frac{\alpha}{2} \ge e(G) +
            \frac{\sqrt{\delta}}{4}x.$$ By \Cref{cor:approx-by-integer-example}
            and Inequality (\ref{eqn:exchange-lower-bd-t-G'}), $e(G') < e$,
            because otherwise there exists a graph $H$ with $n$ vertices, at
            least $e$ edges and more than $\t(n, e)$ non-triangular edges, which
            contradicts the definition of $\t(n, e)$.  Thus $e(G) \le e - \zeta
            x$ for any $\zeta \le \frac{\sqrt{\delta}}{4}$.
        \end{proof}

    \subsection{Compressed graphs} \label{subsec:compressed-graphs}

        We now present the notion of \emph{compressed} graphs.  The proof of
        Motzkin and Straus \cite{motzkin-straus} shows that that it suffices to
        prove Tur\'an's theorem for complete $r$-partite graphs.  Indeed, they
        show (using the notion of weighted graphs) that for every $K_r$-free
        graph $G$ there exists a complete $r$-partite graph with at least as
        many edges as $G$.  The class of compressed graphs will play a similar
        role in this paper as the class of complete $r$-partite graphs in the
        proof of Motzkin and Straus. Compressed graphs have fairly restrictive
        structure (though not quite as simple as complete $r$-partite graphs)
        and we shall see (via \Cref{lem:granulation} below) that it suffices to
        prove \Cref{thm:main} for compressed graphs.  The logarithm in the
        following definition is taken in
        base $2$.

        \begin{defn} \label{defn:compressed} 
            A graph $G$ on $n$ vertices is called \emph{compressed} if the following
            assertions hold.
            \begin{enumerate}
                \item \label{itm:defn-compressed-indep}
                    Every independent set in $G$ is a union of at most $3 \log
                    n$ sets of clones, each one of which, with at most four
                    exceptions, has size at most $3n ^ {1 / 3}$.

                \item \label{itm:defn-compressed-triangular-vs}
                    The set $U$ of \triangularVs{} vertices induces a clique in
                    $G$. Furthermore, the vertices of $U$ all have the same
                    neighbourhood outside of $U$.
            \end{enumerate}
        \end{defn}

        To demonstrate how compressed graphs may be of use to us, we mention the
        following observation.

        \begin{obs} \label{obs:compressed}
            Let $G$ be a compressed graph with $n$ vertices and let $I$ be an
            independent set of size at least $45 n^{1/3} \log n$. Then $I$
            contains a set of clones of size at least $|I| / 5$.
        \end{obs}

        Indeed, let $m$ be the size of the largest set of clones in $I$. Then
        Condition \ref{itm:defn-compressed-indep} of \Cref{defn:compressed}
        implies that $|I| \le 4m + 9n^{1/3} \log n \le 4m + |I| / 5$, so $m \ge
        |I| / 5$.

        The following lemma is the main reason we chose to introduce the notion
        of compressed graphs: it shows that it suffices to prove \Cref{thm:main}
        for compressed graphs.

        \begin{lem} \label{lem:granulation}
            Let $G$ be a graph on $n$ vertices. Then there is a compressed graph
            $H$ such that $|H| = n$,  $e(H) \ge e(G)$ and $\t(H) \ge \t(G)$.
        \end{lem}

        \begin{proof}
            Given a graph $G$ on $n$ vertices, we let $H$ be a weighted graph
            with the following properties.
            \begin{itemize}
                \item $|H| = n$, $e(H) \ge e(G)$ and $\t(H) \ge \t(G)$.
                \item All vertices of $H$ have integer weights.
                \item The number of vertices of $H$ is minimal under the first
                    two conditions.
                \item The number of vertices of weight at least $3 n^{1/3}$ is
                    minimal under the first three conditions.
            \end{itemize}

            We shall show that the graph obtained by replacing each vertex of
            $H$ by a set of clones of size equal to the weight of the vertex, is
            compressed. To that end, we show that $H$ has no independent
            set of size larger than $3 \log n$, and that the vertices with
            weight larger than $3n^{1/3}$ do not contain an independent set of
            size at least five.

            We first show that every independent set of $H$ contains at most $3
            \log n$ vertices.  Suppose to the contrary that $H$ contains an
            independent set $I$ of size $m \ge 3 \log n$. For any set $A
            \subseteq I$ denote $S_A = \sum_{x \in A}\deg(x)$ and $T_A = \sum_{x
            \in A} \degNonTriangular(x)$. Note that $S_A \le n ^ 2$ for every $A
            \subseteq I$.  Since $\binom{m}{m / 2} \ge 2^m / \sqrt{2m} \ge n^3
            / \sqrt{2 n} > n ^ 2$, it follows that there exist
            distinct sets $A, B \subseteq I$ such that $|A| = |B|$ and $S_A =
            S_B$.  By replacing $A$ and $B$ by $A \setminus B$ and $B \setminus
            A$, we may assume that $A \cap B = \emptyset$. Also, without loss of
            generality, $T_A \ge T_B$.
            
            Let $w$ be the minimum weight of a vertex in $B$. Consider the
            weighted graph $H'$, obtained by increasing the weight of each
            vertex in $A$ by $w$ and decreasing the weight of each vertex in $B$
            by $w$ (and removing vertices whose weight becomes $0$). Then $|H'|
            = |H|$, $e(H') = e(H)$, $\t(H') \ge \t(H)$ and the number of
            vertices in $H'$ is smaller than the number of vertices in $H$,
            contradicting the choice of $H$.  It follows that every independent
            set of $H$ contains at most $3 \log n$ vertices.

            We now show that given an independent set of five vertices $\{u_1,
            \ldots, u_5\}$ in $H$ at least one of the vertices $u_i$ has weight
            at most $3n^{1/3}$.  Indeed, suppose that the weight of each of the
            vertices exceeds $3n^{1/3}$.  For any quintuple of non-negative
            integers $k = (k_1, \ldots, k_5)$, denote $S_k = k_1 \deg(u_1)
            + \ldots + k_5 \deg(u_5)$ and $T_k = k_1
            \degNonTriangular(u_1) + \ldots + k_5 \degNonTriangular(u_5)$.
            Consider only the quintuples $k$ that satisfy $k_1 + \ldots + k_5 =
            3 n ^ {1 / 3}$: there are at least $\binom{3 n ^ {1 / 3} + 4}{4} \ge
            \frac{81}{24} n ^ {4 / 3}$ such quintuples and for each of them
            we have $S_k \le 3 n ^ {4 / 3}$.  Thus, there exist distinct
            quintuples $k$ and $l$, whose coordinates are non-negative integers
            whose sum is $3 n ^ {1 / 3}$, such that $S_k = S_l$.  Without loss of
            generality, we may assume that $T_k \ge T_l$.
            
            Consider the graph $H'$, obtained by repeatedly adding $k_i - l_i$ to
            the weight of each vertex $u_i$, as long as all weights remain
            non-negative (note that this process will end because $k_i < l_i$
            for some $i \in [5]$).  The resulting graph $H'$ satisfies $|H'| =
            |H|$, $e(H') = e(H)$ and $\t(H') \ge \t(H)$.  Furthermore, since
            $|k_i - l_i| \le 3 n ^ {1 / 3}$, for some $i \in [5]$ the weight of
            $u_i$ in $H'$ is smaller than $3 n ^ {1 / 3}$.  In particular, $H'$
            has fewer vertices with weight at least $3 n ^ {1 / 3}$ than $H$.
            This is, again, a contradiction to the choice of $H$. It follows
            that every independent set in $H$ has at most four vertices with
            weight at least $3 n ^ {1 / 3}$.

            Recall that $H$ has integer weights, so we may view it as a graph
            where a vertex of weight $w$ represents a set of clones of size $w$.
            The graph $H$ satisfies Condition~\ref{itm:defn-compressed-indep} of
            \Cref{defn:compressed}.  Denote by $U$ the set of \triangularVs{}
            vertices in $H$. We may add all edges missing from $H[U]$ without
            creating new triangles, so we may assume that $U$ induces a clique
            in $H$.  Let $u \in U$ be a vertex of maximum degree in $H$. For
            every $v \in U \setminus \{u\}$, we remove the edges between $v$ and
            $V(H) \setminus U$ and add the edges between $v$ and the
            neighbourhood of $u$ in $V(H) \setminus U$.  This process does not
            decrease the total number of edges and does not create new
            triangles. Moreover, it results in a graph $H'$ that retains
            Condition~\ref{itm:defn-compressed-indep} and whose \triangularVs{}
            vertices have the same neighbourhood outside of $U$.  It follows
            that $H'$ satisfies Conditions \ref{itm:defn-compressed-indep} and
            \ref{itm:defn-compressed-triangular-vs}, i.e., $H'$ is compressed,
            as required.  
        \end{proof}

\section{Almost bipartite} \label{sec:almost-bipartite}

    In this section we prove \Cref{thm:almost-bipartite}.
    
    \thmAlmostBipartite*

    Throughout this section we assume that $G$ is an optimal graph (this means
    that increasing the number of edges reduced the number of non-triangular
    edges and vice versa, see \Cref{def:optimal}) with $n$ vertices and $e = (1
    / 4 + \myeta)n^2$ edges, where $0 < \myeta \le \delta$ and $\delta$ is a
    small fixed positive constant.  Moreover, we always assume that $n$ is large
    enough to satisfy any inequalities that we may write down.
    
    To get a rough idea about how large $t(G)$ is, we derive the following lower
    bound.  Consider the graph $G(a, b, c)$ where $a = \lceil \sqrt{2\myeta} n
    \rceil + 1$, $b = \lceil n/2 \rceil$ and $c = n - a - b = \lfloor n/2
    \rfloor - \lceil \sqrt{2\myeta} n \rceil - 1$. Then $e(G(a,b,c)) =
    \binom{a}{2} + (n-b)b \ge (1/4 + \myeta)n^2$ and $t(G(a,b,c)) = bc \ge \left(1/4
    - \sqrt{\myeta / 2} - O(1/n)\right) n^2$. Since $G$ is optimal, it follows that
    \begin{align} \label{eqn:lower-bd-t-G}
        \t(G) \ge \left( \frac{1}{4} - \sqrt{\frac{\myeta}{2}} -
                  O \left( \frac{1}{n} \right) \right) n^2.
    \end{align}
    Moreover, we have $e \ge \lfloor n^2 / 4 \rfloor + 1$,
    so in fact $\myeta n^2 \ge 1/2$ and therefore $1/n = O(\sqrt{\myeta})$. It
    follows that
    \begin{align} \label{eqn:lower-bd-t-G-alt}
        \t(G) \ge \left( \frac{1}{4} - O\left( \sqrt{\myeta} \right)\right) n^2.
    \end{align}
    
    \remark{should we write $(1/4 + O(\sqrt{\myeta}))$?}

    We divide the proof of \Cref{thm:almost-bipartite} into four parts,
    represented by the following four propositions.  In the first of these
    propositions we show that $G$ has the following structure (see also
    \Cref{fig:struct-almost-bip}), which already shows that $G$ is close to a
    graph $G(a, b, c)$.

    \begin{restatable}{prop}{propStructAlmostBip} \label{prop:struct-almost-bip}
        There is a partition $\{A, B, C, D\}$ of $V(G)$ satisfying the
        following assertions.
        \begin{enumerate}
            \item
                All possible edges between $B$ and $C$ are present in $G$ and
                are non-triangular. In particular, $B$ and $C$ are independent
                sets. Moreover, $|B|, |C| \ge (1 / 2 - O(\sqrt{\myeta}))n$.
            \item
                There are no edges between $A$ and $C$ nor between $B$ and $D$.
            \item
                The induced subgraphs $G[A]$ and $G[D]$ do not have isolated
                vertices.
            \item
                Every vertex in $A \cup D$ is incident with at most
                $O(\sqrt{\myeta}n)$ non-triangular edges of $G$. Moreover, the
                sets $A$ and $D$ do not span non-triangular edges (but there may
                be non triangular edges between $A$ and $D$).
        \end{enumerate}
    \end{restatable}

    \begin{figure}[h]\centering 
        \includegraphics[scale=.8]{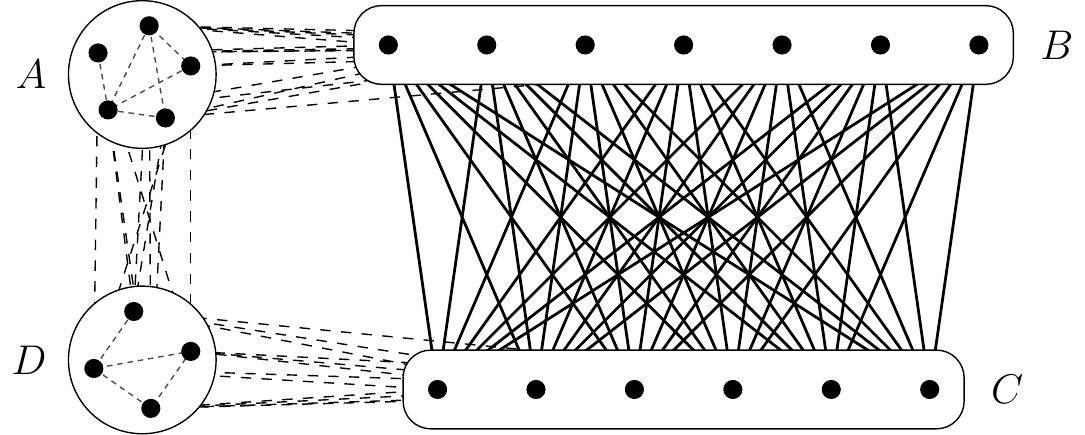}

        \caption{The partition $\{A, B, C, D\}$}
        \label{fig:struct-almost-bip}
    \end{figure}

    From here the proof of \Cref{thm:almost-bipartite} splits into two cases:
    $\myeta \le \myK/n$ and $\myeta \ge \myK/n$, where $\myK$ is a small absolute
    positive constant that will be determined (implicitly) later. The following
    proposition completes the proof when $\myeta$ is small.

    \begin{restatable}{prop}{propFiddlingVeryCloseBip}
        \label{prop:fiddling-very-close-bip}
        Suppose that $\myeta \le \myK/n$ for a sufficiently small absolute
        constant $\myK > 0$. Then $G \cong G(a, b, c)$ for some $a, b, c$.
    \end{restatable}

    \remark{is using the isomorphism sign clearer here?}

    If $\myeta$ is large,  we first obtain sharp estimates for the sizes of
    the sets $A, B, C, D$.

    \begin{restatable}{prop}{propSizesAlmostBip} \label{prop:sizes-almost-bip}
        Let $G$ and $A, B, C, D$ satisfy the conclusions of
        \Cref{prop:struct-almost-bip} and suppose that $|B| \ge |C|$ and that
        $\myeta \ge \myK / n$ for some constant $\myK > 0$.  Then
        \begin{align*}
            |A \cup D| &= \left( \sqrt{2 \myeta} + O_{\myK}(\myeta)\right)n, \\
            |B| &= \left(\frac{1}{2} - O_{\myK}(\myeta ^ {3/4})\right)n, \\
            |C| &= \left(\frac{1}{2} - \sqrt{2 \myeta} + O_{\myK}(\myeta ^ {3
                / 4})\right) n.
        \end{align*}
    \end{restatable}

    The next proposition completes the proof of \Cref{thm:almost-bipartite} in
    the case where $\myeta$ is large.

    \begin{restatable}{prop}{propFiddlingAlmostBip}
        \label{prop:fiddling-almost-bip}
        Suppose that $\myeta \ge \myK / n$ for some absolute
        constant $\myK > 0$. Then $G \cong G(a, b, c)$ for some $a, b, c$.
    \end{restatable}

    \begin{proof} [ of \Cref{thm:almost-bipartite}]
        \Cref{thm:almost-bipartite} immediately follows from
        \Cref{prop:struct-almost-bip,prop:fiddling-very-close-bip,%
        prop:sizes-almost-bip,prop:fiddling-almost-bip}.
        The only minor technicality is that when we replace a graph with at most
        $(1/4 + \delta)n^2$ edges by an optimal graph, we may
        increase the number of edges and lose this condition. However,
        \Cref{lem:exchange} implies that the number of edges can increase by at
        most $O(n)$, so the condition is satisfied with a slightly relaxed
        value of $\delta$. 
    \end{proof}

    The rest of this section is devoted to the proofs of
    \Cref{prop:struct-almost-bip,prop:fiddling-very-close-bip,%
    prop:sizes-almost-bip,prop:fiddling-almost-bip},
    each of which is proved in a separate subsection.
    \subsection{Structure of an optimal graph}

        In this subsection, we prove \Cref{prop:struct-almost-bip} (see also
        \Cref{fig:struct-almost-bip}).

        \propStructAlmostBip*

        The first assertion follows fairly easily from the fact that the number
        of non-triangular edges is almost $n ^ 2 / 4$.  To complete the proof we
        use basic properties of optimal graphs.

        \begin{proof} [ of \Cref{prop:struct-almost-bip}]
            Let $H$ be the subgraph of $G$ whose edges are the non-triangular
            edges of $G$. We note that $H$ is a
            triangle-free graph with close to $n^2 / 4$ edges, which implies that
            $H$ is close to being a complete bipartite graph.
            This enables us to
            find independent (with respect to $G$) sets $U$ and $W$ of size
            almost $n / 2$ each, such that $H$ contains almost all of the
            possible edges between them. 

            \begin{claim} \label{claim:large-almost-complete-bipartite}
                There exist disjoint independent sets $U, W \subseteq V(G)$ such
                that every vertex in $U$ has at least $(1/2 - O(\myeta^{1/4}))n$
                non-triangular neighbours in $W$ and vice versa. In particular,
                $|U|, |W| \ge (1/2 - O(\myeta^{1/4}))n$.
            \end{claim}

            \begin{proof}
                Inequality~(\ref{eqn:lower-bd-t-G-alt}) states that $e(H) =
                \t(G) \ge (1/4 - c\sqrt{\myeta})n^2$ for some absolute constant
                $c$. From this we deduce that there are at most $2d
                \myeta^{1/4}$ vertices in $H$ of degree smaller than $(1/2 -
                d\myeta^{1/4})n$ where $d = \sqrt{c}$.  Indeed, suppose that we
                can find a set $S$ consisting of exactly $2d \myeta^{1 /
                4} n$ vertices of degree
                smaller than \mbox{$(1/2 - d \myeta^{1/4})n$} in $H$.  Since $H$
                is triangle-free, $e(H \setminus S) \le (n - |S|)^2 / 4$. Hence,
                \begin{align*} e(H) \,
                    &< \, \frac{\left( n - |S| \right) ^ 2}{4} + |S| \left(
                    \frac{1}{2} - d\myeta^{1 / 4} \right)n \\
                    &= \, \frac{n ^ 2}{4} - |S| \left( d\myeta^{1 / 4} n -
                        \frac{|S|}{4} \right) \\
                    &= \, \left( \frac{1}{4} - d^2 \sqrt{\myeta} \right) n^2 \\
                    &= \, \left( \frac{1}{4} - c\sqrt{\myeta} \right) n^2,
                \end{align*}
                a contradiction.
                
                Let $u \in V(H)$ be any vertex with $\deg_H(u) \ge (1 / 2 - d 
                \myeta^{1 / 4})n$. Denote by $U$ the set of vertices in $N_H(u)$
                that have at least $(1 / 2 - d \myeta^{1 / 4})n$ neighbours in $H$.
                Since the edges of $H$ are non-triangular in $G$, it follows that
                $U$ is independent in $G$.  Moreover, $|U| \ge \deg_H(u) - 2d
                \myeta^{1/4}n \ge (1/2 - O(\myeta^{1/4}))n$.

                Now let $v \in U$ and denote by $W$ the set of vertices in
                $N_H(v)$ whose degree in $H$ is at least $(1/2 - d 
                \myeta^{1/4})n$. As before, $W$ is independent in $G$ and
                has size at least $(1/2 - O(\myeta^{1/4}))n$. Finally, every
                vertex in $U$ has at least $(1/2 - d \myeta^{1/4})n - (n - |U| -
                |W|) \ge (1/2 - O(\myeta^{1/4}))n$ non-triangular neighbours in
                $W$, and vice versa.
            \end{proof}

            Let $U$ and $W$ be the disjoint independent sets given
            \Cref{claim:large-almost-complete-bipartite}.  The following similar
            claim, allows us to enlarge $U$ and $W$ to obtain sets $B$ and $C$
            which will be shown to satisfy the requirements of
            \Cref{prop:struct-almost-bip}.

            \begin{claim} \label{claim:large-complete-bip}
                There exist disjoint independent sets $B, C \subseteq V(G)$,
                satisfying $U \subseteq B$ and $W \subseteq C$ and 
                 $|B \cup C| \ge (1 - O(\sqrt{\myeta}))n$, such that
                every vertex in $B$ has at least $2n / 5$ non-triangular
                neighbours in $C$ and vice versa.
            \end{claim}
            
            \begin{proof}
                We first show that there are at most $O(\sqrt{\myeta} n)$
                vertices of degree at most $21n / 50$ in $H$. To this end we
                recall Inequality~(\ref{eqn:lower-bd-t-G-alt}), which states
                that $e(H) = \t(G) \ge (1/4 - c \sqrt{\myeta})n$ for some
                absolute constant $c$. Importantly, this constant does not
                depend on $\delta$, so we may choose $\delta$ to satisfy $c
                \sqrt\delta \le 1/100$.  Recall that $\delta$ is an upper bound
                for $\myeta$, so we have $c \sqrt\myeta \le 1/100$.
                
                Suppose that $S$ is a set consisting of exactly $25c\sqrt\myeta
                n$ vertices of degree at most $21n / 50$ in $H$.  Then, similarly
                to the previous claim,
                \begin{align*}
                    e(H) &\le  \frac{\left(n - |S|\right)^2}{4} +
                                 |S|\frac{21n}{50} \\
                         &= \frac{n^2}{4} - |S| \left(\frac{4n}{25} -
                                 \frac{|S|}{4} \right) \\
                         &< \left( \frac{1}{4} - 2c \sqrt\myeta \right) n^2,
                \end{align*}
                a contradiction to Inequality~(\ref{eqn:lower-bd-t-G-alt}).
                Therefore, there are at most $O(\sqrt\myeta n)$ vertices with
                degree at most $21n / 50$ in $H$.
                
                Recall that every vertex in $U$ has at least $(1/2 -
                O(\myeta^{1/4}))n \ge 2n/5$ non-triangular neighbours in $W$ and
                vice versa. Here we implicitly assume that $\delta$ is small
                enough to make this inequality true, and we shall do so
                throughout this proof.

                Denote by $X$ the set of vertices in $V(G) \setminus (U \cup W)$
                whose degree in $H$ is at least $21n/50$. We note that no vertex
                in $X$ has neighbours in both $U$ and $W$. Indeed, suppose
                that $v \in X$ is adjacent to $u \in U$ and $w \in W$. Since $v$
                is not adjacent to any non-triangular neighbour of either $u$ or
                $w$, it has at most $O(\myeta^{1/4}n)$ neighbours in $U$ and at
                most $O(\myeta^{1/4}n)$ neighbours in $W$, implying that
                $\deg_H(v) \le O(\myeta^{1/4}n)$, a contradiction to the
                assumption that $\deg_H(v) \ge 21n/50$.

                Let $Y$ be the set of vertices in $X$ that are adjacent to
                vertices in $U$ and, similarly, let $Z$ be the set of vertices
                in $X$ that have neighbours in $W$.  Then every vertex in $Y$
                has at least $21/50 n - O(\myeta^{1/4} n) \ge 2n/5$
                non-triangular neighbours in $U$ and no neighbours in $W$. In
                particular, since $|U| \le n - |W| < 4n/5$, any two vertices in
                $Y$ share a non-triangular neighbour in $U$, hence $Y$ is an
                independent set in $G$. Denote $B = Y \cup W$ and $C = Z \cup
                U$.
                Then $B$ and $C$ are independent sets, such that every vertex in
                $B$ has at least $2n/5$ non-triangular neighbours in $C$, and
                vice versa. Furthermore, $|B \cup C| = |V(G)
                \setminus X| \ge (1 - O(\sqrt\myeta))n$, so the proof of
                \Cref{claim:large-complete-bip} is complete. 
            \end{proof}

            We can now finish the proof of \Cref{prop:struct-almost-bip}.  Let
            $B$ and $C$ be as in \Cref{claim:large-complete-bip}.  Since every
            vertex in $B \cup C$ has at least $2n / 5$ non-triangular
            neighbours, it follows from \Cref{obs:two-vs} and the assumption
            that $G$ is optimal that every vertex in $G$ has degree at least $2n
            / 5 - 1$.  We conclude (similarly to the proof of
            \Cref{claim:large-complete-bip}) that no vertex in $G$ has neighbours in both
            $B$ and $C$. Indeed, suppose that some $v \in V(G)$ is adjacent to
            some $u \in B$ and $w \in C$. Since $u$ has at least $2n/5$
            non-triangular neighbours in $C$, $v$ is adjacent to at most $|C| -
            2n/5$ vertices in $C$ and, similarly, to at most
            $|B| - 2n/5$ vertices in $B$. It follows that $u$ has degree at most
            $n/5$,
            a contradiction.

            Since no vertex in $G$ is adjacent to a vertex in $B$ and a vertex
            in $C$, we may add all missing edges between $B$ and $C$ without
            creating new triangles. However, $G$ is an optimal graph, so in fact
            all edges between $B$ and $C$ are present in $G$. Again, since
            vertices in $B$ and $C$ do not have common neighbours, all edges
            between $B$ and $C$ are non-triangular.

            We may assume that $|B| \ge |C|$. Then $|B| \ge (1 / 2 -
            O(\sqrtEta))n$ and hence every vertex in $C$ has non-triangular
            degree at least $(1 / 2 - O(\sqrtEta))n$. Again, by
            \Cref{obs:two-vs}, every vertex in $G$ has degree at least $(1 / 2 -
            O(\sqrtEta))n$. Since $B$ is an independent set, it follows that $n
            - |B| \ge (1/2 - O(\sqrt\myeta))n$.  Therefore $|C| = |B \cup C| -
            |B| \ge (1/2 - O(\sqrt\myeta))n$.

            We are now done with the first assertion of
            \Cref{prop:struct-almost-bip}, and the remaining ones follow 
            easily.  Let $A$ be the set of vertices outside of $B \cup C$ that
            are adjacent to a vertex in $B$ and, similarly, let $D$ be the set
            of vertices outside of $B \cup C$ that have a neighbour in $C$.
            Then $\{A, B, C, D\}$ forms a partition of $G$, because a vertex
            without neighbours in $B \cup C$ would have too small a degree. This
            establishes the second assertion.

            To prove the third assertion, we may assume that every vertex in $A$
            has a neighbour in $A$: if some $u \in A$ has no neighbours in $A$,
            then we may add all edges between $u$ and the vertices in $B$
            without creating new triangles and then reassign $u$ to $C$.
            Similarly, we may assume that every vertex in $D$ has a neighbour in
            $D$.

            By inspecting the degrees, any two vertices in $A$ have a common
            neighbour in $B$. Therefore there cannot be any non-triangular edges
            with both ends in $A$ or, similarly, with both ends in $D$.  It
            remains to conclude that every vertex in $A \cup D$ is incident with
            at most $O(\sqrt\myeta n)$ non-triangular edges. Let $u \in A$ and
            let $v \in A$ by a neighbour of $u$. Since $u$ and $v$ have
            neighbours only in $A \cup D \cup B$ and the degree of $v$ is at least
            $(1/2 - O(\sqrt\myeta))n$, it follows that $u$ has at most $|A \cup
            D \cup B| - (1/2 - O(\sqrt\myeta))n = O(\sqrt\myeta n)$
            non-triangular neighbours. The same holds for any vertex
            in $D$. This establishes the fourth assertion and completes the proof
            of \Cref{prop:struct-almost-bip}.
        \end{proof}
        
    \subsection{Completing the proof if $\myeta n$ is small}

        We now prove \Cref{prop:fiddling-very-close-bip}, which
        completes the proof of \Cref{thm:almost-bipartite} in case 
        $\myeta$ is small.

        \propFiddlingVeryCloseBip*

        \begin{proof}
            It follows from the assumptions on the sets $A, B, C, D$ that $|A
            \cup D| = O(\sqrt\myeta n)$ and that each vertex in $A \cup D$ is
            incident with at most $O(\sqrt\myeta n)$ non-triangular edges.
            Therefore the number of non-triangular edges with an end in $A \cup
            D$ is $O(\myeta n^2) = O(\kappa n)$. We show that, in fact, there
            are no such edges.
            
            Suppose that $uv$ is a non-triangular edge with $u \in A \cup D$.
            Without loss of generality, we may assume that $u \in A$ and $v \in
            B \cup D$. Observe that the neighbours of $u$ are not adjacent to
            $v$. Let $G'$ be the graph obtained by adding the edges between $v$
            and the neighbours of $u$ in $A$, removing the edges between $u$ and
            $A \setminus \{u\}$ and also adding all missing edges between $u$
            and $B$. Then $e(G') \ge e(G)$ and $\t(G') \ge \t(G) + |B| -
            O(\kappa n) > \t(G)$, where the last inequality holds provided that
            we choose $\kappa$ small enough.  However, this contradicts the
            optimality of $G$, so there cannot be
            such an edge $uv$.

            It is now easy to finish the proof. By what we have just proved, all
            the missing edges with both ends in $A \cup D$ may be added without
            causing a non-triangular edge to become triangular, i.e., $A \cup D$
            is a clique.  We may assume that $|B| \ge |C|$. Remove the edges
            between $D$ and $C$ and add all possible edges between $D$ and $B$.
            The result is a graph which is isomorphic to $G(|A \cup D|, |B|,
            |C|)$, with at least as many edges and non-triangular edges as $G$.
        \end{proof}

    \subsection{Sizes of $A, B, C, D$}

        In this subsection we prepare for the proof of
        \Cref{thm:almost-bipartite} in the case where $\myeta n$ is large. In
        particular, we obtain good bounds for the sizes of the sets $A \cup
        D$, $B$ and $C$.

        \propSizesAlmostBip*

        The proof is fairly technical, and its main tool is the lower bound on
        $\t(G)$ from Inequality (\ref{eqn:lower-bd-t-G}).

        \begin{proof}
            Denote $a = |A \cup D|$, $b = |B|$ and $c = |C|$ and write
            \begin{align*}
                a &= \left( \sqrt{2\myeta} + \alpha \right) n, \\
                b &= \left( \frac{1}{2} - \beta \right) n, \\
                c &= \left( \frac{1}{2} - \sqrt{2\myeta} + \beta - \alpha
                    \right)n,
            \end{align*}
            where the quantities $\alpha$ and $\beta$ are defined by these
            identities. We cannot assume that $\alpha$ and $\beta$ are positive, but
            we have $-\sqrt{2\myeta} \le \alpha \le O(\sqrt\myeta)$, where the
            second inequality comes from \Cref{prop:struct-almost-bip}.  Since
            there are at most $O(\myeta n^2)$ non-triangular edges with an
            end in $A \cup D$, we have 
            \begin{align*}
                \t(G) \, &\le \, bc + O(\myeta n ^ 2) \\
                         &\le \, \frac{(n-a)^2}{4} + O(\myeta n^2) \\
                         &\le \, \frac{n^2}{4} - \frac{an}{2} + O(\myeta n^2).
            \end{align*}
            Combining this with Inequality~(\ref{eqn:lower-bd-t-G}), which
            states that $\t(G) \ge \left(1/4 - \sqrt{\myeta/2} - O(1/n)\right)n^2$ , we get
            \begin{equation*}
                \frac{1}{4} - \sqrt{\frac{\eps}{2}} - O_\myK(\myeta)
                \; \le \; \frac{\t(G)}{n^2} \; \le \;
                \frac{1}{4} - \frac{(\sqrt{2\myeta} + \alpha)}{2} +
                    O(\myeta).
            \end{equation*}
            Therefore $\alpha = O_\myK(\myeta)$. Using the fact that $b
            \ge c$ and that any vertex in $A \cup D$ sends edges to only one of
            $B$ and $C$, we obtain the following upper bound on the number of
            edges in $G$.
            \begin{align*}
                e(G) &\le b(n - b) + a^2 / 2.
            \end{align*}
            Combining this with the definition $e(G) = (1/4 + \eps)n^2$, we get
            \begin{align*}
                1/4 + \myeta \, 
                &\le \, \left(1/2 - \beta \right)
                \left(1/2 + \beta \right) 
                + \frac{\left( \sqrt{2\myeta} + \alpha \right)^2}{2} \\
                &= \,1/4 - \beta^2 + \myeta + \alpha(\sqrt{2 \myeta} + \alpha/2)
            \end{align*}
            It follows that $\beta \le \alpha(\sqrt{2 \myeta} + \alpha/2)$. In
            particular, $\alpha \ge 0$ and $\beta = O_\myK(\myeta^{3/4})$,
            implying that the assertions of
            \Cref{prop:sizes-almost-bip} hold.
        \end{proof}

    \subsection{Completing the proof if $\myeta n$ is large}

        We are now able to complete the proof of \Cref{thm:almost-bipartite}
        under the assumption that $\myeta \ge \myK / n$ for some absolute
        constant $\myK > 0$.

        \propFiddlingAlmostBip*

        The proof consists of two stages. In the first stage we use the bounds
        from \Cref{prop:sizes-almost-bip} to conclude that $D$ is very small and
        that the number of vertices in $A$ which are adjacent to a
        non-triangular edges is small. In the second stage, we show that if $D$
        is non-empty or if there is a vertex in $A$ with a non-triangular
        neighbour then $G$ can be manipulated to obtain a graph with more edges
        and more non-triangular edges, contradicting the assumption that $G$ is
        optimal. It follows that $G$ is isomorphic to a graph $G(a, b, c)$.

        \begin{proof}[ of \Cref{prop:fiddling-almost-bip}]
            We start by showing that the edges between $B \cup D$ and $A \cup C$
            form an almost complete bipartite subgraph.
            We shall be using the estimates on the size of the sets $A \cup D$,
            $B$ and $C$ from \Cref{prop:sizes-almost-bip}. Note that $\myK$ is an
            absolute constant (determined implicitly in
            \Cref{prop:fiddling-very-close-bip}). Thus we may remove the
            dependence on $\myK$ in the estimates of these sizes.
            \begin{claim} \label{claim:B-A-almost-complete}
                Every vertex in $B \cup D$ is adjacent to all but
                $O(\myeta ^ {3 / 4}n)$ vertices in $A \cup C$.
                Furthermore, $|D| = O(\myeta ^ {3 / 4} n)$.
            \end{claim}

            \begin{proof}
                The non-triangular degree of any vertex in $C$ is at least
                $|B|$. Hence by \Cref{obs:two-vs} every vertex in $G$ has degree
                at least $|B| - 1$. The vertices in $B \cup D$ are not adjacent
                to any vertex in $B$. Since $|B| \ge (1 / 2 - O(\myeta ^ {3 /
                4}))n$, it follows that every vertex in $B \cup D$ is adjacent
                to all but $O(\myeta ^ {3 / 4} n)$ vertices in $V(G) \setminus B
                = A \cup C \cup D$.  Since there are no edges between $B$
                and $D$,  $|D| = O(\myeta ^ {3 / 4} n)$.
            \end{proof}

            Denote by $T$ the set of \triangularVs{} vertices in $A$ (recall
            that a \triangularVs{} vertex is incident only with
            triangular edges) and let $S
            = A \setminus T$. We show that the vertices in $S$ have few
            neighbours in $A$.

            \begin{claim} \label{claim:X-vs-small-deg}
                Every vertex in $S$ has $O(\myeta^{3/4}n)$ neighbours in
                $A$.
            \end{claim}
            
            \begin{proof}
                Let $u \in S$ and let $v$ be a non-triangular neighbour of $u$.
                Then $v \in B \cup D$, because there are no edges between $A$
                and $C$, and there are no non-triangular edges with both ends in
                $A$. Recall that by \Cref{claim:B-A-almost-complete}, $v$ is
                adjacent to all but $O(\myeta^{3/4} n)$ vertices in $A$.
                Since $uv$ is non-triangular, $u$ and $v$ have no common
                neighbours, implying that $u$ has $O(\myeta ^ {3 / 4}
                n)$ neighbours in $A$.
            \end{proof}
            
            We conclude that almost all of the vertices in $A$ are also in $T$.

            \begin{claim} \label{claim:A'-almost-equals-A}
                $|T| \ge (\sqrt{2\myeta} - O(\myeta^{3/4}))n^2$.
            \end{claim}

            \begin{proof}
                By removing the edges with both ends in $A$ or in $D$ from $G$,
                we remain with a bipartite graph, so $(1/4 + \myeta)n^2 =
                e(G) \le n^2/4 + e(G[A]) + e(G[D])$. Since $|D| =
                O(\myeta^{3/4}n)$, we have $e(G[D]) = O(\myeta^{3/2}n^2)$, and
                hence $e(G[A]) \ge (\myeta - O(\myeta^{3/2}))n^2$.

                \Cref{claim:X-vs-small-deg} implies that $e(G[A]) - e(G[T]) \le
                O(|S| \myeta^{3/4}n) = O(|A| \myeta^{3/4}n) =
                O(\myeta^{5/4}n^2)$, where the rightmost inequality is a
                consequence of \Cref{prop:sizes-almost-bip}.  Therefore,
                $e(G[T]) \ge (\myeta - O(\myeta^{5/4})) n^2$, and so $|T| \ge
                (\sqrt{2 \myeta} - O(\myeta^{3/4}))n$, as required.
            \end{proof}
            
            Since $G$ is optimal, we may assume that $T$ induces a clique
            (because the addition of edges to $T$ does not cause a
            non-triangular edge to become triangular). Furthermore, we may
            assume that all vertices in $T$ have the same neighbourhood outside
            of $T$. Indeed, let $u$ be a vertex with largest degree among
            ehe vertices in $T$. We replace $G$ by the graph obtained
            by removing all edges between $T$ and $V(G) \setminus T$, and adding
            all edges between $T$ and the neighbourhood of $u$ (in $G$) outside
            of $T$. This modification does not decrease the number of edges or
            non-triangular edges in $G$.
            
            In particular, if
            a vertex $v \in S$ is adjacent to a vertex in $T$, then
            it is adjacent to all vertices in $T$. However, this is
            impossible since by \Cref{claim:X-vs-small-deg} $v$ has at
            most $O(\myeta^{3/4}n)$ neighbours in $A$, while
            by \Cref{claim:A'-almost-equals-A} there are at least
            $\Omega(\myeta^{1/2}n)$ vertices in $T$. Therefore there are no
            edges between $T$ and $S$.

            In the following claim we deduce that, in fact, the set $S$ is
            empty. The key observation is that a pair of adjacent vertices in
            $S$ can be replaced by one vertex in $C$ and one in $T$, increasing
            both the number of edges and the number of non-triangular edges.
            
            \begin{claim} \label{claim:X-empty}
               The set $S$ is empty.
            \end{claim}

            \begin{proof} 
                Suppose that $S$ contains a vertex $u$. By
                \Cref{prop:struct-almost-bip}, $u$ has a neighbour $v \in A$.
                Since there are no edges between $T$ and $S$, we conclude that 
                $v \in S$. In particular, $u$ and $v$ have no
                neighbours in $T$. Now let $H$ be the graph obtained from $G$ by
                removing the vertices $u$ and $v$ and adding new vertices $x$
                and $y$ where $x$ is joined by edges to $B \cup T$ and $y$ is
                joined to $B$. It follows from
                \Cref{claim:B-A-almost-complete,claim:X-vs-small-deg,%
                claim:A'-almost-equals-A} that $e(H) \ge e(G) -
                O(\myeta^{3/4}n) + (\sqrt{2\myeta} - O(\myeta^{3/4}))n >
                e(G)$.
                Recall that the by \Cref{prop:struct-almost-bip}, the
                non-triangular degree of any vertex in $A$ is at most
                $O(\sqrt{\myeta}n)$, implying that $\t(H) \ge \t(G) -
                O(\sqrt{\myeta}n) + |B| > \t(G)$.
                $H$ has more edges and more non-triangular edges than $G$, a
                contradiction to the assumption that $G$ is optimal. Thus, $S$
                is empty.
            \end{proof}

            Similarly, we prove that $D$ is empty. The trick here is to
            replace two adjacent vertices in $D$ by one vertex in $C$ and one in
            $T$.
            
            \begin{claim} \label{claim:D-empty}
                The set $D$ is empty.
            \end{claim}

            \begin{proof}
                Suppose that $D$ is non-empty, so we may pick adjacent vertices
                $u, v \in D$. Consider the graph $H$, obtained by removing the
                vertices $u$ and $v$ and adding new vertices $x$ and $y$ with
                $x$ joined to $A \cup B$ and $y$ joined to $B$. Note that, since
                $A = T$ is a clique of \triangularVs{} vertices, the addition of
                $x$ and $y$ does not destroy any non-triangular edges in $G
                \setminus \{u, v\}$. It therefore follows from the bounds given
                by \Cref{prop:sizes-almost-bip,claim:B-A-almost-complete} that
                $e(H) \ge e(G) + (\sqrt{2\myeta} - O(\myeta^{3/4}))n > e(G)$.
                Moreover, since $u$ and $v$ each have at most $O(\sqrt{\myeta}n)$
                non-triangular neighbours, $\t(H) \ge \t(G) + (1/2 - O(\sqrt\myeta))n
                > \t(G)$, contradicting the assumption that $G$ is optimal.
            \end{proof}
            
            Now the proof of \Cref{prop:fiddling-almost-bip} is complete.
            Indeed, we know from \Cref{claim:X-empty} that $A = T$. This means
            that $A$ induces a clique and that every vertex in $A$ is adjacent
            to every vertex in $B$. Therefore, $G =
            G(|A|, |B|, |C|)$.
        \end{proof}

\section{Middle range} \label{sec:mid-range}

    In this section we prove \Cref{thm:mid-range}, in which we consider the case
    where the graph is neither close to being complete nor close to being
    complete bipartite.
    Out of the three ranges, the middle range turns out to be the
    hardest to prove.  One of the main difficulties that arises here is that,
    unlike the other two ranges, we cannot directly conclude that the graph is
    close to a graph $G(a, b, c)$.
    
    \thmMidRange*

    Fix $\delta > 0$. Throughout this section we assume that $G$ is a compressed
    and optimal graph with $n$ vertices and $e$ edges, where $(1/4 + \delta) n^2
    \le e \le (1/2 - \delta) n^2$. Moreover, whenever we write down an
    inequality that holds for large $n$, we assume that $n$ is large enough to
    satisfy it.
    
    We split the proof of \Cref{thm:mid-range} into four stages, as described by
    the four following propositions.  In the first stage we show that $G$ has
    many \triangularVs{} vertices.

    \begin{restatable}{prop}{propLargeClique} \label{prop:large-clique}
        $G$ has $\Omega(n)$ \triangularVs{} vertices.
    \end{restatable}

    In the second stage we conclude that $G$ admits the following structure (see
    also \Cref{fig:struct-mid-range}). This implies that $G$ vaguely resembles a
    graph $G(a, b, c)$.

    \begin{restatable}{prop}{propStructure} \label{prop:structure}
        There is a partition $\{A, B, C\}$ of $V(G)$ such that all parts have
        size $\Omega(n)$ and the following properties are satisfied.

        \begin{enumerate}
            \item \label{cond:struct-A}
                $A$ is the set of \triangularVs{} vertices in $G$, it spans a
                clique and its vertices are adjacent to all of $B$ and none of
                $C$.

            \item \label{cond:struct-B}
                $B$ may be partitioned into $O(1)$ sets of clones and a
                remainder of size $O(n^{1/2} \log n)$.

            \item \label{cond:struct-C}
                $C$ may be partitioned into $O(1)$ sets of clones, each
                having $\Omega(n)$ non-triangular neighbours in $B$,
                and a remainder of size $O(n^{1 / 3} \log n)$.
        
        \end{enumerate}
    \end{restatable}

    \begin{figure}[h]\centering 
        \includegraphics[scale=1.2]{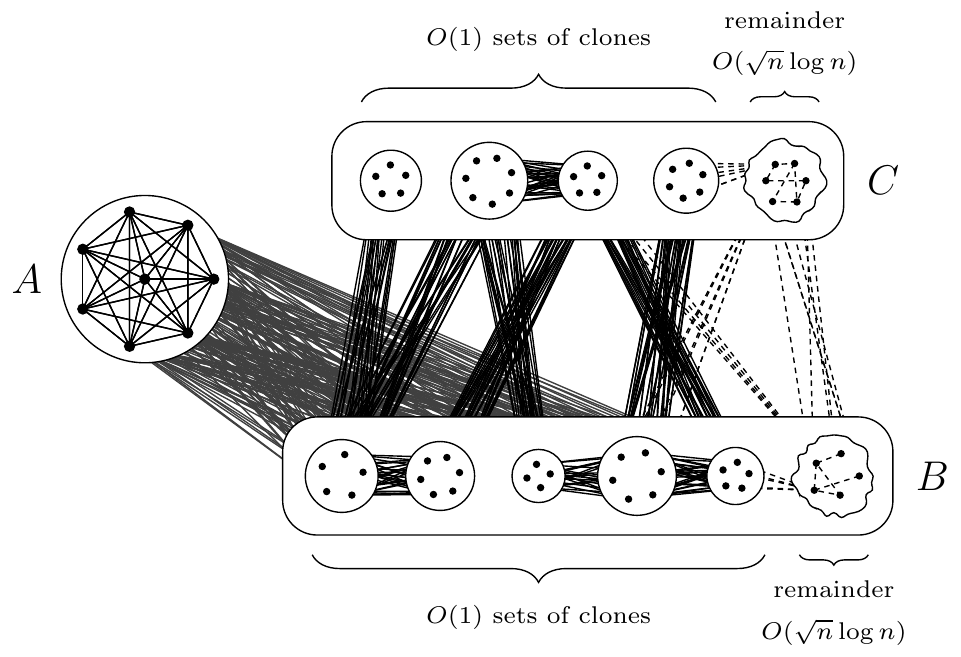}

        \caption{The partition $\{A, B, C\}$}
        \label{fig:struct-mid-range}
    \end{figure}
    
    In the third stage we show that the number of edges (and non-triangular
    edges) in $G$ is close to the number of edges (and non-triangular edges) in
    $G(|A|, |B|, |C|)$.

    \begin{restatable}{prop}{propSizes} \label{prop:sizes}
        Let $A, B, C$ be as in \Cref{prop:structure} and denote $a = |A|, b =
        |B|, c = |C|$.  Then $e(G) = a^2/2 + ab + bc + O(ֿ\approxBedge)$  and
        $\t(G) = bc + O(\approxBedge)$.  
    \end{restatable}

    In the final fourth stage we complete the proof of \Cref{thm:mid-range}.
   
    \begin{restatable}{prop}{propFiddling} \label{prop:fiddling}
        $G \cong G(a, b, c)$ for some $a, b, c$.
    \end{restatable}

    \begin{proof}[ of \Cref{thm:mid-range}]
        The proof is immediate from
        \Cref{prop:large-clique,prop:structure,prop:sizes,prop:fiddling}.
        The only slight technicality is that when we replace a graph with at
        most $(1/2 - \delta) n^2$ edges by an optimal and compressed graph, the
        number of edges may increase and exceed this bound. However,
        \Cref{lem:exchange} implies that the number of edges can increase by at
        most $O(n)$, so the condition is still satisfied for a relaxed value of
        $\delta$.
    \end{proof}

    We now turn to the proofs of
    \Cref{prop:large-clique,prop:structure,prop:sizes,prop:fiddling}. We present
    them in separate subsections.

    \subsection{Many \triangularVs{} vertices}

        In this subsection we prove \Cref{prop:large-clique}.

        \propLargeClique*

        The main ingredient of this proof is a surprising application of the
        exchange lemma, \Cref{lem:exchange}, and the assumption that $G$ is
        compressed. First, we conclude from
        \Cref{lem:approx-by-good-weighted-graph} that $G$ has a large clique.
        Then, we partition the graph into fairly large independent sets of clones
        and a very dense part, using the fact that $G$ is compressed. It is then
        possible to conclude that only few of the vertices of the clique are
        adjacent to non-triangular edges.
        \remark{should we still call it the exchange lemma?}

        \begin{proof} [ of \Cref{prop:large-clique}]
            Our first aim is to show that $G$ has a clique of size at least
            $\Omega(n)$. This can be done fairly easily, as
            shown in the proof of the following claim.

            \begin{claim} \label{claim:large-clique}
                $G$ has a clique of size $\Omega(n)$.
            \end{claim}

            \begin{proof}
                By \Cref{lem:approx-by-good-weighted-graph}, there exists a
                \good{}
                weighted subgraph $H$ of $G$ satisfying $|H| = |G| = n$, $e(H) \ge e(G)$,
                $\t(H) \ge \t(G)$ (see \Cref{def:good} for the definition of a
                \good{} weighted graph). Let $\{K, \{u, v\}\}$ be a partition of
                $V(H)$ into a clique $K$ and an edge $uv$, which is the only
                non-triangular edge in $H$.

                Let $\alpha$ be the sum of the weights of the vertices in $K$
                and let $m$ be the number of vertices in $K$. Let $\beta$ and
                $\gamma$ be the weights of $u$ and $v$ and suppose that $\beta
                \ge \gamma$. Note that $\alpha + \beta + \gamma = n$. By the
                Cauchy-Schwarz inequality, the contribution of the vertices in
                $K$ towards $e(H)$ is maximised if all of these vertices have
                weight $\alpha / m$. Therefore this contribution does not
                exceed $(\alpha/m)^2 \binom{m}{2} = (1 - 1/m) \alpha^2 / 2$.
                Moreover, since no vertex is adjacent to both $u$ and $v$,
                the contribution of
                the edges between $K$ and $\{u, v\}$ towards $e(H)$ is maximised
                when every vertex in $K$ is adjacent to $u$, but not $v$.
                Hence,
                \begin{equation} \label{eqn:bound-on-edges}
                    e(G) \le e(H) \le
                    \left( 1 - \frac{1}{m} \right) \frac{\alpha^2}{2} +
                    \alpha \beta + \beta \gamma.
                \end{equation}

                In particular, since $\beta \gamma \le n^2/4$, we have $e(G) \le
                n^2 / 4 + \alpha n$. Recall that $e(G) \ge (1/4 + \delta)n^2$.
                It follows that $\alpha \ge \delta n$.

                Denote $b = \lceil \beta \rceil$, $c = \lceil \gamma \rceil$ and $a =
                n - b - c$ and consider the graph $F = G(a, b, c)$.
                Note that $\t(G) \le \t(H) = \beta \gamma \le bc = \t(F)$.
                Since $G$ is optimal, it follows that $e(G) \ge e(F)$.
                Therefore,
                \begin{align*}
                    e(G) \ge e(F)  
                        &= \binom{a}{2} + ab + bc \\ 
                        &\ge \frac{(\alpha - 2)(\alpha - 3)}{2} + (\alpha - 2)\beta +
                            \beta \gamma \\
                        &\ge \frac{\alpha^2}{2} + \alpha \beta + \beta \gamma -
                            2.5n \\
                            &= \left( 1 - \frac{5n}{\alpha^2} \right)
                                \frac{\alpha^2}{2} + \alpha \beta + \beta \gamma.
                \end{align*}
                
                By (\ref{eqn:bound-on-edges}) we have $m \ge \frac{\alpha^2}{5n} \ge
                \frac{\delta^2}{5} n$. It follows that $G$ has a clique of size
                at least $\frac{\delta ^ 2}{5} n$.
            \end{proof}
           
            Recall that $G$ is compressed.
            It follows from \Cref{defn:compressed} that every independent
            set in $G$ of size at least $5 n ^ {1 / 2}$ contains a set of clones of
            size at least $n ^ {1 / 2}$ (see \Cref{obs:compressed}).
            
            We construct a set $U \subseteq V(G)$ as follows.  We start with $U
            = \emptyset$.  At each stage, if the complement $\comp{U} = V(G)
            \setminus U$ contains an independent set $I$ of size at least $5 n ^
            {1 / 2}$, then $I$ contains a set of clones of size at least $n ^ {1
            / 2}$. We add this set of clones to $U$ and continue until
            $\comp{U}$ has no independent set of size at least $5 n ^ {1 / 2}$.
            Observe that $U$ is a disjoint union of sets of clones each of size
            at least $n ^ {1 / 2}$, while the complement $\comp{U}$ has no
            independent set of size at least $5 n ^ {1 / 2}$ (see
            \Cref{fig:set-U}).

            \begin{figure}[h]\centering
                \includegraphics{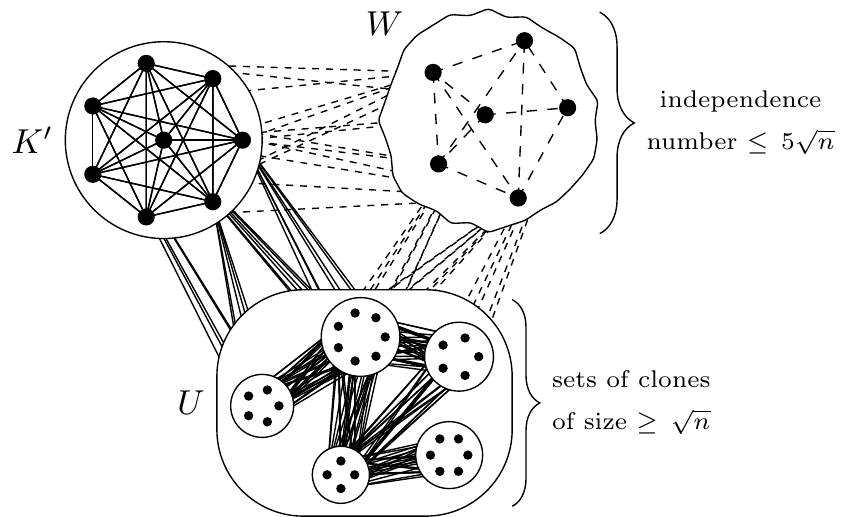}
                \caption{The sets $U$, $W$ and $K'$}
                \label{fig:set-U}
            \end{figure}

            In the following claim we deduce from \Cref{lem:exchange} that
            $G[\comp{U}]$ is very dense.

            \remark{complement sign very ugly}

            \begin{claim} \label{claim:bound-missing-edges}
                $G[\comp{U}]$ has $O(n^{3/2})$ non-edges. 
            \end{claim}

            \begin{proof}
                Since $G[\comp{U}]$ has no independent set of size at least $5n
                ^ {1 / 2}$, every vertex in $G$ has at most $5n ^ {1 / 2}$
                non-triangular neighbours in $\comp{U}$. It follows that the
                number of non-triangular edges with at least one end in
                $\comp{U}$ is at most $5 n ^ {3 / 2}$. 
                
                Denote by $m$ the number of non-edges in $G[\comp{U}]$. By
                adding these edges to $G$ we obtain a graph $G'$ with $n$
                vertices and $e(G) + m$ edges such that $t(G') \ge \t(G) - 5 n ^
                {3/2}$. It follows from \Cref{lem:exchange} that $m = O(n ^ {3 /
                2})$.
            \end{proof}

            Let $K$ be a largest clique in $G$, so $|K| = \Omega(n)$ by
            \Cref{claim:large-clique}. Let $K' = K \setminus U$ and denote $W =
            \comp{U} \setminus K'$ (see \Cref{fig:set-U}). Note that,
            since $U$ contains no clique of size greater than $n ^ {1 / 2}$, we
            have $|K'| \ge |K| - n ^ {1 / 2} = \Omega(n)$.  In the following
            claim we use the structure of $U$ and the previous claim to deduce
            that almost all the vertices in $K'$ are \triangularVs, i.e.~are
            incident with triangular edges only.

            \begin{claim} \label{claim:few-non-triangular-vs-K}
                All but $O(n ^ {1 / 2})$ vertices in $K'$ are
                \triangularVs.
            \end{claim}

            \begin{proof}
                Since $K'$ is a clique, any vertex in the complement $V(G)
                \setminus K'$ sends at most one non-triangular edge to $K'$. In
                fact, if $u \in V(G) \setminus K'$ has a non-triangular
                neighbour $v \in K'$ then $u$ has no neighbours in $K'$ except
                for $v$.  
                
                Denote by $m$ the number of vertices in $W$ that have a
                non-triangular neighbour in $K'$. Then the number of missing
                edges in $G[\comp{U}]$ is at least $m (|K'| - 1) = \Omega(mn)$.
                From \Cref{claim:bound-missing-edges} we conclude that $m = O(n
                ^ {1 / 2})$. This implies that the number of vertices in $K'$
                that have a non-triangular neighbour in $\comp{U}$ is at most
                $O(n ^ {1 / 2})$.

                Finally, $U$ is a union of at most $n ^ {1 / 2}$ sets of clones,
                and any one set of clones can send non-triangular edges to at
                most one vertex in $K'$.  Therefore there are at most $n ^ {1 /
                2}$ vertices in $K'$ with a non-triangular neighbour in $U$.
            \end{proof}

            The clique $K'$ is of size $\Omega(n)$ and we now know that all but
            $O(n^{1/2})$ vertices in $K'$ are \triangularVs. It follows that $G$
            has $\Omega(n)$ \triangularVs, completing
            the proof of \Cref{prop:large-clique}.
        \end{proof}    

    \subsection{Structure}

        In this subsection we build on the fact that $G$ has $\Omega(n)$
        \triangularVs{} vertices and prove that, in terms of structure, $G$ is
        not far off from being isomorphic to a graph $G(a, b, c)$. In
        particular, we prove that the vertices of $G$ can be partitioned into
        three linearly sized sets $A, B, C$ such that $A$ is a clique, and all
        edges between $A$ and $B$ are present in $G$, while all edges between
        $A$ and $C$ are missing. We do not yet prove that the sets $B, C$ are
        independent, but we show that both of them can be partitioned into a
        small number of independent sets (see \Cref{fig:struct-mid-range}). Our
        main tool in this subsection is the assumption that $G$ is compressed,
        we also use \Cref{lem:exchange}.

        \propStructure*

        \begin{proof}
            Denote by $A$ the set of \triangularVs{} vertices in $G$. Recall
            that $G$ is compressed, hence by \Cref{defn:compressed}, the
            vertices of $A$ have the same neighbourhood outside of $A$. Denote
            this neighbourhood by $B$ and let $C = V(G) \setminus (A \cup B)$.
            Condition \ref{cond:struct-A} follows.

            Note that the graph $G(a, b, c)$, where $c = \frac{\delta}{2}n$, $b
            = \sqrt{\delta}n$ and $a = n - b - c$, has at least $(1 / 2 -
            \delta)n^2$ edges and $\frac{\delta^{3/2}}{2}n^2$ non-triangular
            edges. Hence, since $G$ is optimal and $e(G) \le (1/2 - \delta)n^2$,
            it follows that $\t(G) = \Omega(n^2)$.

            By \Cref{prop:large-clique} we have $|A| = \Omega(n)$.  Note that
            there are no non-triangular edges with both ends in $B$, so the
            number of non-triangular edges in $G$ is at most $|C|(|B| + |C|)$.
            Since $\t(G) = \Omega(n ^ 2)$, it follows that $|C| = \Omega(n)$.
            We will deduce that $|B| = \Omega(n)$ from a stronger statement that
            almost all vertices in $C$ have $\Omega(n)$ non-triangular
            neighbours in $B$.

            \begin{claim}
                All but $O(1)$ vertices of $C$ have $\Omega(n)$ non-triangular
                neighbours in $B$.
            \end{claim}

            \begin{proof}
                Let $c > 0$ and $k \in \mathbb{N}$ be constants.  Suppose that
                there is a set $Z \subseteq C$ of size $k$ whose every vertex
                has at most $cn$ non-triangular neighbours in $B$.  Our aim is
                to show that if $c$ is sufficiently small and $k$ is
                sufficiently large, then no such set $Z$ exists.
                
                Consider the graph $G'$, obtained from $G$ by adding the edges
                between $Z$ and $A$.  Then $e(G') = e(G) + k |A|$ and $t(G') \ge
                \t(G) - ckn - \binom{k}{2} \ge \t(G) - 2ckn$. Provided that $k$ is
                sufficiently large, \Cref{lem:exchange} implies that $\t(G') \le
                \t(G) - \zeta k |A|$ for some constant $\zeta > 0$ that does not
                depend on $c$ or $k$. Therefore $\zeta |A| \le 2cn$ must hold.
                However, we may choose $c$ small enough to make this false, thus
                obtaining a contradiction.
            \end{proof}
            
            Denote by $C'$ the set of vertices in $C$ that have $\Omega(n)$
            neighbours in $B$. The previous claim implies that $|C \setminus C'|
            = O(1)$. The following claim implies that $C'$ may be partitioned
            into $O(1)$ independent sets.
            
            \begin{claim} \label{claim:chromatic-number-C}
                There exists a set $S \subseteq B$ of size $O(1)$ such that
                every vertex in $C'$ has a non-triangular neighbour in $S$.
            \end{claim}

            \begin{proof}
                We construct $S = \{u_1, \dotsc, u_k\}$ by selecting the
                elements $u_1, \dotsc, u_k \in B$ and subsets $I_1, \ldots, I_k
                \subseteq B$ in the following way. Suppose
                that $u_1, \ldots, u_j$ and $I_1, \ldots, I_j$ have been
                selected, where $j \ge 0$.  Consider the set $U_j$ consisting of
                vertices in $C'$ that have
                a non-triangular neighbour in $\{u_1, \ldots, u_j\}$ (so in
                particular $U_0 = \emptyset$). If $U_j = C'$, we stop the
                process.  Otherwise, pick a vertex $v \in C' \setminus U_j$ and
                consider the set $N$ consisting of the non-triangular neighbours
                of $v$ in $B$. By the definition of $C'$, we have $|N| =
                \Omega(n)$. Moreover, since $N$ is independent and $G$ is
                compressed, $N$ contains a set of clones of size at
                least $|N| / 5$. Denote this set of clones by $I_{j+1}$ and pick
                $u_{j+1} \in I_{j+1}$ arbitrarily.
                
                It is clear that when the process terminates, every vertex in
                $C'$ has a non-triangular neighbour in the resulting set $S$.
                It remains to check that the process stops after $O(1)$ steps.
                Indeed, suppose that it ran for $k$ steps.  The sets $I_1,
                \ldots, I_k$ are pairwise disjoint and have size at least
                $\Omega(n)$ each, whence $k = O(1)$. 
            \end{proof}

            The non-triangular neighbourhoods of the vertices in $S$ cover $C'$.
            Therefore $C'$ can be partitioned into $O(1)$ independent sets.
            Since $G$ is compressed, each independent set can be partitioned
            into $O(\log n)$ sets of clones, all but at most four of which have
            size $O(n ^ {1 / 3})$. Therefore $C'$ can be partitioned into $O(1)$
            sets of clones and a remainder of size $O(n ^ {1/3} \log n)$ (the
            $O(1)$ sets of clones being the four largest sets of clones within
            each independent set in the decomposition of $C'$, and the remainder
            consisting of the remaining $O(\log n)$ sets of clones, each of
            which has size at most $O(n ^ {1/3})$). Note that, by definition,
            every
            vertex in $C'$ has $\Omega(n)$ non-triangular neighbours in $B$. Now
            throw all of the $O(1)$ vertices of $C \setminus C'$ into the
            remainder to get a partition of $C$ that satisfies Condition
            \ref{cond:struct-C}.
            
            It remains to prove Condition \ref{cond:struct-B}. Let $Z$ denote
            the remainder in the partition of $C$ and denote $C'' = C \setminus
            Z$. Let $Y$ be the set of vertices in $B$ that do not have
            non-triangular neighbours in $C''$ and denote $B' = B \setminus Y$.
            We will show that $|Y| = O(n^{1/2} \log n)$ and that $B'$
            can be partitioned
            into $O(1)$ independent sets (from which it follows as before that
            $B'$ can be partitioned into $O(1)$ sets of clones and a
            remainder of size $O(n^{1/3} \log n)$).

            \begin{claim} \label{claim:small-leftover-B}
                $|Y| = O(n ^ {1 / 2} \log n)$.
            \end{claim}

            \begin{proof}
                Recall that $A$ is the set of \triangularVs{} vertices in $G$.
                Therefore every vertex in $B$ has a non-triangular neighbour,
                and that neighbour must be in $C$.  In particular, the
                non-triangular neighbourhoods of $Z$ cover $Y$.  Since $Z$ is a
                union of $O(\log n)$ sets of clones, this implies that $Y$ can
                be partitioned into $O(\log n)$ independent sets. In particular,
                $Y$ contains an independent set $I$ of size $\Omega(|Y| / \log
                n)$.

                Let $G'$ be the graph obtained from $G$ by adding all possible
                edges spanned by $|I|$. Then $e(G') = e(G) + \binom{|I|}{2}$ and
                $\t(G') \ge \t(G) - |I||Z| \ge \t(G) - O(|I| n^{1/3} \log n)$.
                This is a contradiction to \Cref{lem:exchange} unless
                $\binom{|I|}{2} = O(n)$ or $\binom{|I|}{2} = O(|I| n ^ {1 /
                3} \log n)$. In either case $|I| = O(n ^ {1 / 2})$ and so $|Y| =
                O(n ^ {1 / 2} \log n)$.
            \end{proof}

            To complete the proof of \Cref{prop:structure}, it remains to show
            that $B'$ can be partitioned into $O(1)$ independent sets. This is
            immediate upon recalling that $C''$ is the union of $O(1)$ sets of
            clones and that every vertex of $B'$ has a non-triangular neighbour
            in $C''$. Indeed, $B'$ is the union of the non-triangular
            neighbourhoods of the vertices in $C''$, and we have
            $O(1)$ such neighbourhoods.  
        \end{proof}

    \subsection{Sizes}

        In the previous subsection we proved that $V(G)$ can be partitioned into
        sets $A, B, C$ that correspond to the three parts of a graph $G(|A|,
        |B|, |C|)$. In this subsection we consider the sizes of the sets $A, B,
        C$.  We show that the number of edges (and non-triangular edges) of $G$
        is very close to the number of edges (and non-triangular edges) of
        $G(|A|, |B|, |C|)$.

        \propSizes*

        In the proof of this proposition we revisit F\"{u}redi and Maleki's
        proof \cite{furedi-maleki} of \Cref{thm:furedi-maleki} which is an approximate
        version of our main theorem. We simulate their proof of
        \Cref{lem:approx-by-good-weighted-graph}, but keep tight control on the
        order of vertices to which we apply
        \Cref{lem:process-no-empty-triangle}.
        \remark{better explanation?}

        \begin{proof}[ of \Cref{prop:sizes}]
            Recall that by \Cref{prop:structure} each of the sets $B$ and $C$ can be
            partitioned into $O(1)$ sets of clones and a remainder of size
            $O(\approxAvx)$.  Let $G'$ be the graph obtained by removing the
            edges adjacent to vertices in these remainders.  Then $e(G') \ge
            e(G) - O(\approxAedge)$ and $\t(G') \ge \t(G) - O(\approxAedge)$. 

            The following claim is a variation of
            \Cref{lem:approx-by-good-weighted-graph}. It allows us to
            approximate $G'$ by a weighted graph that induces a clique on $C$.

            \begin{claim} \label{claim:clique-on-C}
                There is a weighted subgraph $H$ of $G'$ such that $|H| =
                n$, $e(H) \ge e(G')$ and $\t(H) \ge \t(G')$. Furthermore, $H$
                has the following properties.
                
                \begin{itemize}
                        
                        \item
                            All vertices in $A$ that are present in $H$,
                            with at most one exception, have weight $1$.
                        
                        \item
                            All vertices in $B$ are present in $H$ and have
                            weight $1$.

                        \item
                            The vertices in $C$ that are present in $H$ induce a
                            clique.

                \end{itemize}
            \end{claim}

            \begin{proof}
                We perform the following process to obtain the weighted graph
                $H$. Initially, we set $H$ to be $G'$ with every vertex being
                given weight $1$. Then we perform multiple steps, during which
                we modify the weights of the vertices in $A \cup C$ (and remove
                some of these vertices) so that at any given time $A$ has at
                most one vertex with weight other than $1$. At each step we
                select vertices $u \in A$ and $v, w \in C$. We take $u$ to
                be the unique vertex in $A$ of weight not equal to $1$, and if
                there is no such vertex, we take it to be an arbitrary vertex
                remaining in $A$. We take $v$ and $w$ to be any pair of
                non-adjacent (in $G'$) vertices in $C$.
                If choosing $u,
                v, w$ according to these rules is impossible, then we terminate
                the process.

                Suppose that we successfully selected the vertices $u, v, w$. We
                may remove one or two of them and redistribute their weight onto
                the remaining ones according to
                \Cref{lem:process-no-empty-triangle} so that the new weights are
                still positive, the total weight does not change and neither
                $e(H)$ nor $\t(H)$ decrease.  It is clear that this process
                terminates, because each step decreases the number of vertices
                remaining in $H$.

                Let us consider the resulting weighted graph $H$. Since the
                process terminated, either no vertices of $A$ are present in
                $H$, or the remaining vertices of $C$ induce a clique. We show
                that the latter condition must hold.

                Suppose that all vertices of $A$ were removed from $H$. Denote
                by $m$ the size of the largest clique that can be formed from
                vertices remaining in $H$. Since the vertex set of $G'$ can be
                partitioned into $A$ and $O(1)$ independent sets, we have $m =
                O(1)$.  Apply \Cref{lem:approx-by-good-weighted-graph} to obtain
                a \good{} weighted subgraph $F$ of $H$, with $xy$ being its only
                non-triangular edge, such that $|F| =
                n$, $e(F) \ge e(G')$ and $\t(F) \ge \t(G')$.
                Let $\beta$ and $\gamma$ be the weights of $x$ and $y$ in $F$
                and suppose 
                that $\beta \ge \gamma$. Then $\alpha = n - \beta - \gamma$ is
                the sum of the weights of the other vertices in $F$.  We have
                $\t(F) = \beta \gamma$ and, as in
                Inequality~(\ref{eqn:bound-on-edges}) from
                \Cref{claim:large-clique}, $e(F) \le (1 - 1/m) \alpha^2 / 2 +
                \alpha \beta + \beta \gamma$.  It follows that $\t(G) \le \beta
                \gamma + O(n ^ {3 / 2} \log n)$ and $e(G) \le \alpha^2/2 +
                \alpha \beta + \beta \gamma - \Omega(n^2)$.  Consider the graph
                $G'' = G(n - \ceil{\beta} - \ceil{\gamma}, \ceil{\beta},
                \ceil{\gamma})$. It is easy to check that $\t(G'') \ge t(G) -
                O(n ^ {3 / 2} \log n)$ and $e(G'') \ge e(G) + \Omega(n ^ 2)$.
                This is a contradiction to \Cref{lem:exchange},
                since $G$ is optimal.
                
                It follows that the set of vertices in $C$ that are present in
                $H$ induces a clique. Hence, the weighted graph $H$ satisfies
                the requirements of \Cref{claim:clique-on-C}.
            \end{proof}

            Let $H$ be a weighted graph as given by \Cref{claim:clique-on-C}, so
            in particular, $e(H) \ge e(G) - O(n^{3/2} \log n)$ and $\t(H) \ge
            \t(G) + O(n^{3/2} \log n)$.
            By \Cref{lem:exchange}, since $G$ is optimal,
            \begin{align} \label{eqn:number-edges-H-close-to-G}
                \begin{split}
                    e(H) &= e(G) + O\big(n^{3/2} \log n\big), \\
                    \t(H) &= \t(G) + O\big(n^{3/2} \log n\big).
                \end{split}
            \end{align}
            We remark that these two lines express both upper and lower bounds
            for the quantities $e(H)$ and $\t(H)$.
            In the following claim we prove that, in fact, only one vertex of $C$
            is present in $H$.
            
            \begin{claim} \label{claim:structure-of-H}
                Exactly one vertex of $C$ is present in $H$. Moreover, all but
                at most $O(\approxBvx)$ vertices of $B$ are non-triangular
                neighbours of that vertex.
            \end{claim}

            \begin{proof}
                Denote by $u_1, \ldots, u_m$ the vertices in $C$ that are
                present in $H$, and let $N_1, \dotsc, N_m$ be their
                non-triangular neighbourhoods in $B$. Since the set $\{u_1,
                \dotsc, u_m\}$ forms a clique, there are no edges between
                $u_i$ and $N_j$ for $i \neq j$. In particular, the sets $N_1,
                \dotsc, N_m$ are pairwise disjoint.

                Let $Z$ be the set of vertices in $B$ that have no
                non-triangular neighbours in $H$, that is, $Z = B \setminus (N_1
                \cup \ldots \cup N_m)$. We will show that $|Z| = O(\approxBvx)$.
                Indeed, recall that $B$ is the union of $O(1)$ independent sets
                and a remainder of size at most $O(\approxAvx)$. Thus if $|Z|
                \ge \Omega(\approxBvx)$, there is an independent set $I \subseteq
                Z$ of size $\Omega(|Z|)$. Consider the weighted graph $H'$
                obtained from $H$ by adding the edges spanned by $I$.  Then
                $e(H') = e(H) + \Omega(|Z|^2) \ge e(G) - O(n^{3/2} \log n) +
                \Omega(|Z|^2)$ and $\t(H') = \t(H) \ge \t(G) - O(n ^ {3 / 2}
                \log n)$.  It follows from \Cref{lem:exchange} that $|Z| = O(n ^
                {3 / 4} \sqrt{\log n})$.

                \begin{figure}[h]\centering
                    \includegraphics{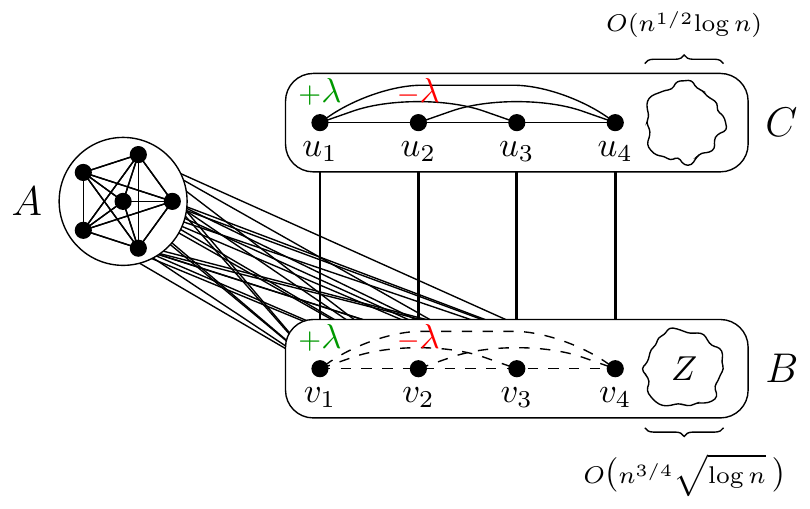}
                    \caption{The graph $H'$}
                    \label{fig:weight-shifting}
                \end{figure}

                Our aim is to prove that $m = 1$. We assume for contradiction
                that $m \ge 2$. In particular, since in $G'$ every
                vertex in $C$ is either isolated or has $\Omega(n)$
                non-triangular neighbours in $B$, the vertices $u_1, \dotsc,
                u_m$ have the latter property.  In other words, $|N_i| =
                \Omega(n)$ for every $i$.
                
                For each $i$, let $\gamma_i$ denote the weight of $u_i$ in $H$.
                We will show that $\gamma_i = \Omega(n)$ for every $i$. Indeed,
                fix any $i$. Denote by $H_i$ the weighted graph obtained from
                $H$ by adding all edges spanned by $N_i$. Since $N_i$ is an
                independent set in $H$, we have $e(H_i) \ge e(G) + \Omega(n ^
                2)$ and $\t(H_i) \ge \t(G) - |N_i| \gamma_i - O(\approxAedge)$.
                By \Cref{lem:exchange}, $\gamma_i = \Omega(n)$.

                Write $\beta_i = |N_i|$. Construct a weighted
                graph $F$, starting from $H$ and carrying out the following
                steps. Firstly, remove all edges with an end in $Z$. Secondly,
                replace each set $N_i$ by a vertex $v_i$ of weight $\beta_i$.
                Finally, connect each vertex $v_i$ to all of the vertices in $A$
                (that are present in $H$) as well as to $u_i$ and $v_j$ for
                every $j \neq i$. We have $e(F) \ge e(G) - O(\approxBedge)$ and
                $\t(F) \ge \t(H) \ge \t(G) - O(\approxAedge)$.

                Pick any real $\lambda$ such that $|\lambda| \le \min \{ \beta_1,
                \beta_2, \gamma_1, \gamma_2 \}$. Let $F_\lambda$ be the weighted
                graph obtained from $F$ by adding $\lambda$ to the weights of
                $u_1$ and $v_1$ and subtracting $\lambda$ from the weights of
                $u_2$ and $v_2$. Clearly, $|F_\lambda| = |F| = n$ and
                it is easy to check that $e(F_\lambda) = e(F)$.
                
                If $m \ge 3$, then the only non-triangular edges in $F$ are $u_i
                v_i$. Hence, in this case,
                \begin{align*}
                    \t(F_\lambda)  
                        & = \t(F) - (\beta_1 \gamma_1 + \beta_2 \gamma_2)
                                  + (\beta_1 + \lambda)(\gamma_1 + \lambda)
                                  + (\beta_2 - \lambda)(\gamma_2 - \lambda) \\
                        & = \t(F) + (\beta_1 + \gamma_1 - \beta_2 - \gamma_2)\lambda
                                  + 2\lambda ^ 2.
                \end{align*}
                If $\beta_1 + \gamma_1 \ge \beta_2 + \gamma_2$, then take
                $\lambda = \min\{ \beta_1, \gamma_1, \beta_2, \gamma_2 \}$.
                Otherwise, take $\lambda = -\min\{ \beta_1, \gamma_1, \beta_2,
                \gamma_2 \}$. In either case, $|\lambda| = \Omega(n)$ and
                $\t(F_\lambda) \ge \t(F) + \Omega(n ^ 2) \ge \t(G) + \Omega(n ^
                2)$, contradicting \Cref{lem:exchange}.

                This calculation is slightly different in the case when $m = 2$,
                because then we have to account for the edge $u_1 u_2$, which is
                also non-triangular. In this case
                \begin{align*}
                    \t(F_\lambda) 
                    & = \t(F) - (\beta_1 \gamma_1 + \beta_2 \gamma_2 +
                    \gamma_1 \gamma_2) + (\beta_1 + \lambda)(\gamma_1 + \lambda) + (\beta_2 -
                    \lambda)(\gamma_2 - \lambda) + (\gamma_1 + \lambda)(\gamma_2
                    - \lambda) \\ 
                    & = \t(F) + (\beta_1 - \beta_2)\lambda + \lambda ^ 2.
                \end{align*}
                We may reach a contradiction to \Cref{lem:exchange} by choosing
                $\gamma$ to be of the same sign as $\beta_1 - \beta_2$ and 
                $|\lambda| = \min\{\beta_1, \beta_2, \gamma_1, \gamma_2\}$.
                We conclude that $m = 1$, completing the
                proof of the claim.
            \end{proof}

            Recall that all but one vertex in $A$ has weight either $0$ or $1$
            in $H$. In fact, we may assume that there are no vertices in $H$
            with weight $0$. In the following claim we show that the remaining vertex
            cannot have very large weight.

            \begin{claim} \label{claim:weight-special-vx-A}
                If some vertex in $A$ has weight (in $H$) other than $1$, then
                that weight is at most $O(\approxBvx)$.  
            \end{claim}

            \begin{proof}
                Let $u$ be a vertex of $A$ of maximal weight in $H$, and let
                $\omega$ be its weight. Suppose that $\omega > 1$, in which case
                all other vertices in $A$ have weight $1$ in $H$.

                Replace the vertex $u$ by a clique of size $\floor{\omega}$
                whose vertices have weight $\omega / \floor{\omega}$ and are
                adjacent to all of $(A \setminus \{u\}) \cup B$ and denote the
                resulting weighted graph by $H'$. We have to check the technical
                condition that the average weight of a vertex in $H'$ is at
                least $1$. However, this can be easily verified, since the total
                weight of $H'$ is an integer and $H'$ has at most one vertex
                whose weight is smaller than $1$ (namely, the only vertex of $C$
                that remains in $H'$).

                We have $\t(H') = \t(H) \ge \t(G) - O(n ^ {3 / 2} \log n)$ and
                $e(H') \ge e(G) - O(n^{3/2} \log n) + \Omega(\omega^2) $.  It
                follows from \Cref{lem:exchange} that $\omega = O(\approxBvx)$.
            \end{proof}

            Recall that $a, b, c$ are the sizes of the sets $A, B, C$ in the
            original graph $G$. Let $\alpha, \beta, \gamma$ be the sums of
            weights (in $H$) of the vertices in these sets, summing over the
            vertices that are present in $H$. So, for example, $\beta = b$ and
            $\gamma$ is the weight of the single vertex in $C$ that is present
            in $H$. Clearly, $\alpha + \gamma = a + c$, because both sides are
            equal to $n - b$. Now we use the properties of $H$ that we have
            proved to get good bounds on $e(H)$
            and $\t(H)$ in terms of $\alpha, \beta, \gamma$.

            Recall that the set $A$ induces a clique in $G$, so its
            remainder induces a clique in $H$. Combined with
            \Cref{claim:weight-special-vx-A}, this implies that the contribution
            of the edges within $A$ to $e(H)$ is $\alpha^2 / 2 - O(n^{3/2} \log
            n)$. By \Cref{claim:structure-of-H}, the set $B$ contains an
            independent set of size at least $|B| - O(\approxBvx)$. Therefore
            the contribution of the edges within $B$ to $e(H)$ (and in
            particular to $\t(H)$) is $O(n^{7/4} \sqrt{\log n})$. Moreover,
            \Cref{claim:structure-of-H} implies that the edges between $B$ and
            $C$ contribute $\beta \gamma - O(n^{7/4} \sqrt{\log n})$ to both
            $e(H)$ and $\t(H)$. Putting this together, we get
            \begin{align} \label{eqn:number-edges-H}
                \begin{split}
                    e(H) &= \alpha^2 / 2 + \alpha \beta + \beta \gamma +
                        O\big(n^{7/4} \sqrt{\log n}\big) \\ 
                    \t(H) &= \beta \gamma + O\big(n^{7/4} \sqrt{\log n}\big). 
                \end{split}
            \end{align}
            Again, we remark that these are both upper and lower bounds for the
            quantities $e(H)$ and $\t(H)$. We deduce that $\alpha$ almost equals
            $a$ and $\gamma$ almost equals $c$.

            \begin{claim} \label{claim:size-of-C}
                $\alpha = a + O(\approxBvx)$ and $\gamma = c + O(\approxBvx)$.
            \end{claim}

            \begin{proof}
                We can read the inequality $\alpha \le a + O(\approxBvx)$ off
                \Cref{claim:weight-special-vx-A}. To get the upper bound on
                $\alpha$, we consider the quantity $e(H) - \t(H)$. On one hand,
                the inequalities in (\ref{eqn:number-edges-H}) give
                \begin{equation*}
                    e(H) - \t(H) = \alpha^2/2 + \alpha \beta + O\big(n^{7/4}
                    \sqrt{\log n}\big).
                \end{equation*}
                On the other hand, we can use the inequalities in
                (\ref{eqn:number-edges-H-close-to-G}) to get
                \begin{align*}
                    e(H) - \t(H)
                        &= e(G) - \t(G) + O(n^{3/2} \log n) \\
                        &\ge a^2/2 + ab + O(n^{3/2} \log n),
                \end{align*}
                where the latter inequality comes from the fact that the
                quantity $e(G) - \t(G)$ counts the triangular edges in $G$,
                and all vertices in $A$ are triangular.  Recall that $b =
                \beta$.  Combining the two inequalities for $e(H) - \t(H)$ we
                get $\alpha \ge a + O(\approxBvx)$, so $\alpha = a +
                O(\approxBvx)$. To complete the
                proof of the claim, note that $a + c = \alpha + \gamma$.
            \end{proof}
            
            \Cref{prop:structure} follows from the last claim and the
            inequalities in (\ref{eqn:number-edges-H}).
        \end{proof}

    \subsection{End of the proof}

        We are now ready to complete the proof of \Cref{thm:mid-range}.

        \propFiddling*

        We gradually get closer to proving that $G \cong G(a, b, c)$. We start by
        showing that $b$ is much bigger than $c$, which leads to the conclusion
        that $C$ spans no non-triangular edges. This implies that almost all
        possible edges between $B$ and $C$ are present in $G$ and are
        non-triangular, by \Cref{prop:sizes}. In fact, using the fact that $G$
        is compressed, we deduce that there are large subsets of $B$ and of
        $C$ that span a complete bipartite graph consisting of non-triangular
        edges. With some more effort, using the optimality of $G$, we 
        conclude that $B$ and $C$ themselves induce a complete bipartite graph,
        thus completing the proof.

        \begin{proof}[ of \Cref{prop:fiddling}]
            Let $\{A, B, C\}$ be the partition of $V(G)$ given by
            \Cref{prop:structure}. As in the statement of \Cref{prop:sizes},
            write $a = |A|, b = |B|, c = |C|$.  We start by showing that $b$ is
            significantly larger than $c$.

            \begin{claim} \label{claim:b-larger-than-c}
                We have $b \ge c + \Omega(n)$.
            \end{claim}

            \begin{proof}
                Suppose to the contrary that $b \le c + o(n)$. Then we have $n =
                a + b + c \ge 2b + a - o(n)$, and hence $b \le (n - a +
                o(n))/2$. Since $a = \Omega(n)$, we can conclude that $b \le n/2
                - \Omega(n)$.

                Consider the graph $H = G(a, b, c)$. \Cref{prop:sizes}
                implies that $e(H) = e(G) + O(n^{7/4} \sqrt{\log n})$ and $\t(H)
                = \t(G) + O(n^{7/4} \sqrt{\log n})$. Consider also the
                graph $H' = G(a, b + d, c - d)$, where $d = \floor{n^{1.999}}$.
                Note that $b + d \le n/2 - \Omega(n)$.  Therefore, from the
                expression $e(H) = \binom{a}{2} + (n-b)b$ and the corresponding
                expression for $e(H')$, we can see that $e(H') \ge e(H) +
                \Omega(dn) = e(G) + \Omega(dn)$. Similarly, $\t(H') \ge \t(H) -
                o(dn) = \t(G) - o(dn)$. However, this contradicts
                \Cref{lem:exchange}, so $b = c + \Omega(n)$.
            \end{proof}

            In the following claim we conclude that $C$ spans no non-triangular
            edges.

            \begin{claim} \label{claim:in-C-only-triangular}
                There are no non-triangular edges with both ends in $C$.
            \end{claim}

            \begin{proof}
                By \Cref{prop:sizes} there are $bc + o(n^2)$ non-triangular
                edges in $G$, and each one of them is incident with a vertex in
                $C$. Therefore some vertex in $C$ has at least $b - o(n)$
                non-triangular neighbours. Thus, by \Cref{obs:two-vs},
                every vertex in $C$ has degree at least $b - o(n)$, so 
                the sum of the degrees of any two vertices in $C$ is at least $2b -
                o(n) > b + c = |B \cup C|$, where the latter inequality comes from the
                previous claim. Since the vertices in $C$ have neighbours only
                in $B \cup C$, it follows that any pair of vertices in $C$ have a
                common neighbour, hence they cannot be joined by a
                non-triangular edge.
            \end{proof}

            Recall that by \Cref{prop:structure} each of the sets $B$ and $C$ can be
            partitioned into $O(1)$ sets of clones and a remainder of size
            $O(n^{1/2} \log n)$. In such a partition of $B$ consider the sets of
            clones of size at least $n^{9/10}$ and let $B'$ be their union.
            Similarly, let $C'$ be the union of the sets of clones in the
            partition of $C$ that have size at least $n^{9/10}$ and denote $Z =
            C \setminus C'$ and $Y = B \setminus B'$. Then $|Y| = O(n^{9/10})$
            and $|Z| = O(n^{9/10})$. We show that all possible edges between
            $B'$ and $C'$ are present in $G$ and are non-triangular.
            
            \begin{claim} \label{claim:B'-C'-complete}
                All possible edges between $B'$ and $C'$ are present in $G$ and
                are non-triangular.
            \end{claim}

            \begin{proof}
                From the previous claim we know that every non-triangular edge
                in $G$ has one end in $B$ and one end in $C$.  Suppose that
                there is a pair of vertices, one in $B'$ and one in $C'$, that
                are not joined by a non-triangular edge. Then there are two sets
                of clones of size at least $n^{9/10}$, one contained in $B'$ and
                the other in $C'$, between which there are no non-triangular
                edges. But then $\t(G) \le bc - n^{9/5}$, contradicting
                \Cref{prop:sizes}.
            \end{proof}

            In the following claim we obtain additional information about $Y$
            and $Z$, thereby we get close to showing that $B$ and $C$ induce a complete
            bipartite graph.

            \begin{claim} \label{claim:B'-C'-indep}
                There are no edges between $B'$ and $Y$ and between
                $C'$ and $Z$. Moreover, every vertex in $B$ has at
                least $c - o(n)$ neighbours in $C$, and every vertex in $C$ has
                at least $b - o(n)$ neighbours in $B$.
            \end{claim}

            \begin{proof}
                Any vertex in $C'$ has at least $|B'| = b - o(n)$ non-triangular
                neighbours, so \Cref{obs:two-vs} implies that every vertex in $G$
                has degree at least $b - o(n)$. 
                \Cref{claim:b-larger-than-c} implies that $|C| + |B \setminus
                B'| \le c + o(n)$ is smaller than this quantity, hence every vertex in
                $C$ has a neighbour in $B'$. Therefore, because all possible
                edges between $B'$ and $C'$ are present in $G$ and are non-triangular (by
                \Cref{claim:B'-C'-complete}), there are no edges between $C'$
                and $Z$. In particular, every vertex in $C$ has at
                most $o(n)$ neighbours in $C$, so it has at least $b - o(n)$
                neighbours in $B$.

                Pick any vertex $u \in Y$. Since $u$ is not in
                $A$, $u$ has a non-triangular neighbour $v \in C$.  We have just
                proved that $v$ has at least $b - o(n)$ neighbours in $B'$.
                Therefore, $u$ has at most $o(n)$ neighbours in $B$.  Now
                suppose that $u$ is adjacent to a vertex in $B'$. Then
                $u$ has no neighbours in $C'$. Hence, $u$ has at most $a +
                o(n)$ neighbours, of which at most $o(n)$ are non-triangular.
                However, any vertex of $B'$ has at least $a + c - o(n)$
                neighbours and at least $c - o(n)$ of them are non-triangular.
                This contradicts \Cref{obs:two-vs}.

                It remains to verify that every vertex in $Y$ has at
                least $c - o(n)$ neighbours in $C$. Let $w \in Y$
                and denote by $d$ the number of neighbours of $w$ in $C$.
                Then the degree of $w$ is at most $a + d + o(n)$ and its
                non-triangular degree is at most $d$. By 
                \Cref{obs:two-vs}, applied to $w$ and any vertex in $B'$, we know
                that $a + d \ge a + c - o(n)$ or $d \ge c - o(n)$. In either
                case $d \ge c - o(n)$.
            \end{proof}

            In the following claim we prove that no edges are spanned by $Z$.
            We use a trick that we have used several times before, replacing a
            pair of adjacent vertices in $Z$ by copies of vertices in $A$ and
            $B'$, thus increasing both the number of edges and of triangular
            edges.

            \begin{claim} \label{claim:C-indep}
                $Z$ spans no edges.
            \end{claim}

            \begin{proof}
                Suppose to the contrary that $Z$ contains a pair of
                adjacent vertices $u, v$. All of the non-triangular neighbours
                of $u$ are in $B$ and they are not adjacent to $v$. By
                \Cref{claim:B'-C'-indep}, $v$ has $b - o(n)$ neighbours in
                $B$, and hence $u$ has at most $o(n)$ non-triangular neighbours.
                Likewise, $v$ has at most $o(n)$ non-triangular neighbours.

                Consider the graph $G'$ obtained from $G$ by removing $u,v$ and
                adding new vertices $x$ and $y$, where $x$ is joined by edges to
                all vertices in $A \cup B$, and $y$ is joined to all vertices
                in $B'$ (in other words, $u$ and $v$ are replaced by two new vertices,
                one in $A$ and one in $B'$). We have $e(G') \ge e(G) + a - o(n)$
                and $\t(G') \ge \t(G) + b - o(n)$. This contradicts the
                optimality of $G$, because $a = \Omega(n)$ and $b = \Omega(n)$.
            \end{proof}

            A similar trick enables us to conclude that $Y$ spans no edges, but
            here we replace all non-isolate vertices in $Y$ with suitable copies
            of other vertices.
        
            \begin{claim} \label{claim:B-indep}
                There are no edges with both ends in $Y$.
            \end{claim}

            \begin{proof}
                Let $S$ be the set of vertices in $Y$ that have a
                neighbour in $Y$. Our aim is to prove that $S =
                \emptyset$. To this end, write $m = |S|$ and suppose that $m
                \neq 0$. Then $2 \le m = o(n)$. Let $G'$ be the graph
                obtained from $G$ by removing all vertices in $S$, adding
                $\floor{m/2}$ new vertices to $A$ and adding $\ceil{m/2}$ new
                vertices to $B'$. More precisely, we add $\floor{m/2}$ new
                vertices adjacent to all of $A \cup B'$ and $\ceil{m/2}$ more
                new vertices adjacent to all of $A \cup C'$.

                Let us compare $e(G')$ and $\t(G')$ with $e(G)$ and $\t(G)$. It
                follows from \Cref{claim:B'-C'-complete,claim:B'-C'-indep} that
                if an edge in $G$ has both ends in $B$, then in fact it has 
                both ends in $S$. In particular, the removal of $S$ decreases
                the total number of edges by at most $(a + c)m + \binom{m}{2} =
                (a + c + o(n))m$. On the other hand, the addition of the new
                vertices increases this number by $(a + c - o(n))m + (b -
                c)\floor{m/2} \ge (a + c + \Omega(n))m$. Therefore, $e(G') \ge
                e(G) + \Omega(mn) > e(G)$.

                Moreover, in $G$, every vertex in $S$ has at most $o(n)$
                non-triangular neighbours, because every vertex in $S$ is
                adjacent to another vertex in $S$, which has at least $c - o(n)$
                neighbours in $C$. Therefore, the removal of $S$ decreases the
                number of non-triangular edges by at most $o(mn)$. On the other
                hand, the addition of new vertices increases this number by $(b
                - o(n)) \ceil{m/2} = \Omega(mn)$. Therefore, $\t(G') > \t(G)$
                and this contradicts the optimality of $G$. This means that $m =
                0$, so $Y$ spans no edges.
            \end{proof}

            \Cref{prop:fiddling} easily follows from
            \Cref{claim:B'-C'-complete,claim:B'-C'-indep,claim:C-indep,claim:B-indep}.
            Indeed, these claims together imply that $B$ and $C$ are independent sets
            in $G$. Therefore, if there were missing edges between $B$ and $C$,
            we could add them to $G$ without creating new triangles.
            Since $G$ is an optimal graph, all possible edges between
            $B$ and $C$ are present. It follows that $G \cong G(|A|, |B|, |C|)$.
        \end{proof}

\section{Almost complete} \label{sec:almost-complete}

    In this section we prove \Cref{thm:almost-complete}.
    
    \thmAlmostComplete*

    The proof in this range is easier than the other two ranges, though far from
    immediate. We start by proving that $G \cong G(a, b, c)$ for some $a, b, c$
    if there are only very few (at most $2n - 8$) missing edges. For the
    remaining range, we obtain a partition of the vertices into sets $A, B, C$
    according to the degrees of the vertices and aim to show that $G \cong
    G(|A|, |B|, |C|)$.
    We first prove that the sets have
    the correct orders of magnitude using a rough lower bound on $\t(G)$. We are
    then able to prove better estimates for the size of the sets, and, finally,
    we deduce that $G$ has the required structure.

    \begin{proof}[ of \Cref{thm:almost-complete}]

        Fix a small constant $\delta > 0$ (whose
        value can be worked out from the proof) and let $G$ be an optimal
        graph with $n$ vertices and $\binom{n}{2} - \myeta n^2$
        edges, where $0 \le \myeta \le \delta$. 
        We may assume that the set of \triangularVs{} induces a clique and all
        its vertices have the same neighbourhood among the remaining vertices
        (see e.g.~\Cref{lem:granulation}).
        As usually, we assume that $n$
        is large enough to satisfy all inequalities that we write down in the
        proof.

        We first consider the case $e(G) \ge \binom{n}{2} - (2n-9)$. 

        \begin{claim}
            If $e(G) \ge \binom{n}{2} - (2n - 9)$ then $G \cong G(a, b, c)$ for
            some $a, b, c$.
        \end{claim}
        
        \begin{proof}
            If $G$
            has no non-triangular edges, then by optimality, it is a clique, so
            we are done ($G = G(n, 0, 0)$).  We claim that $G$ does not have two
            independent non-triangular edges.  Indeed, if $uv$ and $xy$ are such
            edges, then for any other vertex $w$ one of the two possible edges
            $uw$ and $vw$ is missing, as well as one of $xw$ and $yw$. Therefore
            $G$ has at least $2n - 8$ missing edges, contradicting our
            assumption.  Therefore, since the triangle-free edges cannot form a
            triangle, they form a star. Let $uv_1, \dotsc, uv_k$ be the
            non-triangular edges. Then the set $A = V(G) \setminus \{u, v_1,
            \dotsc, v_k\}$ is the set of \triangularVs{} vertices in $G$, so $A$
            induces a clique and all of the vertices in $A$ have the same
            neighbourhood in $V(G) \setminus A$. So there are two possibilities:
            either $u$ is adjacent to all of $A$, or $u$ is not adjacent to any
            vertex in $A$. In the former case, there are no edges between $A$
            and $\{u_1, \dotsc, u_k\}$, so $G \cong G(|A|, 1, k)$.  In the
            latter case, optimality of $G$ implies that all possible edges
            between $A$ and $\{u_1, \dotsc, u_k\}$ are present in $G$, so $G
            \cong G(|A|, k, 1)$.
        \end{proof}

        From this point onwards we assume that $e(G) \le \binom{n}{2} -
        (2n-8)$.  In particular, $\myeta \ge (2 - o(1))/n$. We wish to prove
        that $G$ is isomorphic to the graph $G(a, b, c)$ for some parameters
        $a, b, c$. It is easy to see that these parameters should be
        approximately equal to $b \approx \sqrt{2 \myeta / 3} n$ and $c
        \approx (2 \myeta / 3) n$, because this is when $bc$ is maximised
        subject to conditions $n = a + b + c$ and $cn + \binom{b}{2} \le
        \myeta n$. We use this observation to get a lower bound for $\t(G)$.

        \begin{claim}
            $\t(G) = \Omega(\myeta^{3/2} n^2)$.
        \end{claim}

        \begin{proof}
            Let $G' = G(a, b, c)$, where $b = \floor{\sqrt{2\myeta / 3} n}$, $c
            = \floor{(2 \myeta / 3)n}$ and $a = n - b - c$. Then the number of
            edges missing from $G'$ is at most $\binom{b}{2} + cn \le \myeta n$.
            We now find a lower bound for $\t(G')$, but we need to be
            careful with the rounding-down errors in the definitions of $b$ and
            $c$.

            Recall that $\myeta n^2 \ge 2n - 9$, implying that
            $(2\myeta/3)n \ge 4/3 - 6/n$, hence $c = \floor{(2\myeta/3)n} \ge
            \frac{1}{5} (2\myeta/3)n = (2\myeta/15) n$.
            Since $\myeta n^2 \ge 2n - 9$, we have $\sqrt{\myeta} n \ge \sqrt{2n
            - 9}$, hence, very crudely, $b \ge (\sqrt{\myeta/3})n$.
            It follows that $\t(G') = bc \ge \lambda\myeta^{3/2} n^2$, for
            $\lambda = 2/(15\sqrt{3})$.
            Since $G$ is optimal, it follows that $\t(G) \ge \t(G') =
            \Omega(\myeta^{3/2} n^2)$.
        \end{proof}

        We now define three sets $A, B, C \subseteq V(G)$ that correspond to the
        three parts of a graph $G(a, b, c)$. Let $C$ be the set of vertices of
        degree at most $3n/4$; let $B$ be the set of vertices in $V(G)
        \setminus C$ that have a non-triangular neighbour in $C$; and let
        $A = V(G) \setminus (B \cup C)$. Since any two vertices in $A \cup B$
        have at least $n/2$ common neighbours, there are no non-triangular edges
        in $A \cup B$.  Therefore, all vertices in $A$ are \triangularVs{}, so
        $A$ induces a clique and its vertices have the same neighbourhood in
        $V(G) \setminus A$.

        The next step is to obtain tight bounds for the sizes of $A, B, C$.
        First, we determine the order of magnitude of $|B|$ and $|C|$.

        \begin{claim}
            $|B| = \Theta( \sqrt\myeta n)$ and $|C| = \Theta(\myeta n)$.
        \end{claim}

        \begin{proof}
            By definition, every vertex in $C$ is an end of at least $n/4$
            non-edges. Since there are $\myeta n^2$ non-edges in total,
            we have $|C| = O(\myeta n)$. We know from the previous claim that
            there are at least $\Omega(\myeta^{3/2} n^2)$ non-triangular edges.
            All of these edges have at least one end in $C$, and so some vertex
            in $C$ has at least $\Omega(\sqrt\myeta n)$ non-triangular
            neighbours. By \Cref{obs:two-vs}, every vertex in $G$ has at least
            $\Omega(\sqrt\myeta n)$ neighbours.

            Pick any $v \in B$. By the definition of $B$, $v$ has a
            non-triangular neighbour $u \in C$. This means that $v$ is not
            adjacent to any of the neighbours of $u$, so $v$ is an end of at
            least $\Omega(\sqrt\myeta n)$ non-edges. Therefore, $|B| =
            O(\sqrt\myeta n)$. Moreover, since every non-triangular edge has
            both ends in $C$, or one in $B$ and one in $C$, we have $|B||C| +
            |C|^2/2 \ge \Omega(\myeta^{3/2} n^2)$, which implies that $|B| =
            \Omega(\sqrt\myeta n)$ and $|C| = \Omega(\myeta n)$.
        \end{proof}

        An immediate consequence of the previous claim is that $|A| = (1 -
        O(\sqrt\myeta)) n$. Recall that all vertices in $A$ have the same
        neighbourhood in $V(G) \setminus A$. In particular, each vertex in $B
        \cup C$ is adjacent either to all vertices in $A$ or to none of them.
        Since the vertices in $B$ have degree at least $3n/4$, they are all
        adjacent to all of $A$, and, similarly, there are no edges between $A$
        and $C$. We can use the latter fact to give a better upper
        bound for $C$.

        \begin{claim}
            There is a constant $\xi > 0$ such that $\xi \myeta n \le |C| \le (1
            - \xi) \myeta n$.
        \end{claim}

        \begin{proof}
            The lower bound follows from the previous claim. To prove the upper
            bound, note that every vertex in $B$ is an end of $\Omega(\sqrt
            \myeta n)$ missing edges and $|B| = \Theta(\sqrt \myeta n)$, so
            there are $\Omega(\myeta n^2)$ missing edges with an end in $B$.
            Since all edges between $A$ and $C$ are missing, we have $(1 -
            O(\sqrt\myeta)) n |C| + \Omega(\myeta n^2) \le \myeta n^2$.
            Therefore, $|C| \le (1 - \Omega(1)) \myeta n / (1 -
            O(\sqrt\myeta))$, and the claim follows provided that $\myeta$ is
            sufficiently small.
        \end{proof}

        It is now possible to accurately relate the sizes of $B$ and $C$.  Write
        $|C| = \gamma \myeta n$, where $\xi \le \gamma \le 1 - \xi$.  Define
        $\beta = \sqrt{2(1-\gamma)}$ and note that $\beta = \Theta(1)$, by 
        the previous claim.
    
        \begin{claim} \label{claim:B-size-and-non-triangular}
            $|B| = \beta \sqrt\myeta n + O(\myeta n)$. Moreover, there are at
            least $|B||C| - O(\myeta^2 n^2)$ non-triangular edges between $B$
            and $C$.
        \end{claim}

        \begin{proof}
            Let $G'$ be the graph $G(a, b, c)$, where $c = |C|$, $b =
            \floor{\beta \sqrt\myeta n}$ and $a = n - b - c$. It is easy to see
            that $b^2/2 + cn \le \myeta n^2$.
            In particular, we have $e(G') \ge \binom{n}{2} - \myeta n^2 = e(G)$.
            Therefore, since $G$ is optimal, $\t(G) \ge \t(G') = bc$.

            Let us again restrict our attention to $G$. Since every
            non-triangular edge has an end in $C$, some vertex in $C$ has at
            least $b$ non-triangular neighbours. It follows from
            \Cref{obs:two-vs} that every vertex in $G$ has degree at least $b -
            1$.
            Moreover, since any vertex in $C$ is adjacent only to
            vertices in $B$ and $C$, we have $|B| \ge b - c - 1 = b - O(\myeta
            n)$.

            Every vertex of $B$ has a non-triangular neighbour, and therefore is
            non-adjacent to at least $b - 1$ vertices. Hence, there are at least
            $|B|(b-1)/2$ missing edges with an end in $B$. Since there are no
            edges between $A$ and $C$, we have
            \begin{equation*}
                \frac{1}{2} |B|(b-1) + \left( 1 - O(\sqrt\myeta) \right) cn
                \;\le\; \myeta n^2
                \;\le\; \frac{1}{2} b^2 + cn + O(\sqrt\myeta n),
            \end{equation*}
            where the latter inequality follows from the definition of $b$.
            It follows that $|B| \le b + O(\myeta n)$. To finish the proof, observe
            that $\t(G) \ge bc = |B||C| - O(\myeta^2 n^2)$, and recall that the
            non-triangular edges of $G$ are either spanned by $C$ (but there are
            $O(\myeta^2 n^2)$ such edges) or they have one end in $B$ and the
            other in $C$.
        \end{proof}

        A standard trick, which we have been using throughout the paper, allows
        us to conclude that $C$ is an independent set.

        \begin{claim} \label{claim:C-indep-dense-case}
           $C$ is independent. 
        \end{claim}

        \begin{proof}
            The second conclusion of \Cref{claim:B-size-and-non-triangular}
            implies that some vertex in $C$ has at least $|B| - O(\myeta n)$
            non-triangular neighbours in $B$. As a consequence, $B$ contains an
            independent set $I$ of size $|B| - O(\myeta n)$. Moreover,
            \Cref{obs:two-vs} implies that every
            vertex in $C$ is adjacent to all but at most $O(\myeta n)$ vertices
            in $B$.

            Suppose that $C$ contains a pair of adjacent vertices $u, v$. Let
            $G'$ be the graph obtained from $G$ by removing the vertices $u$ and
            $v$ and adding new vertices $x$ and $y$ where $x$ is adjacent to all
            of $A \cup B$, and $y$ is adjacent to all of $I$. The removal of $u$
            and $v$ decreases the total number of edges by at most $2(|B| + |C|)
            = O(\sqrt\myeta n)$, while the addition of $x$ and $y$ increases this
            number by at least $|A| = (1 - O(\sqrt\myeta)) n$. Therefore, $e(G')
            > e(G)$. Moreover, since $u$ and $v$ are adjacent, they do not form
            non-triangular edges with their common neighbours.  Hence, $u$ and
            $v$ have at most $O(\myeta n) + |C| = O(\myeta n)$ non-triangular
            neighbours in total.  On the other hand, the addition of $x$ and $y$
            adds $|I| = \Omega(\sqrt\myeta n)$ non-triangular edges. Therefore,
            $\t(G') > \t(G)$, a contradiction to the optimality of $G$.
        \end{proof}

        Finally, we prove that $B$ is an independent set.

        \begin{claim}
            $B$ is independent.
        \end{claim}

        \begin{proof}
            Suppose that $u, v \in B$ are adjacent. There are
            at most $|C|$ non-triangular edges with an end in $\{u, v\}$,
            because every vertex can only be a non-triangular neighbour of at
            most one
            of $u$ and $v$. Moreover, by definition, every vertex in $B$ has a
            non-triangular neighbour. Let $w \in C$ be a non-triangular
            neighbour of $u$.  Since the edge $uw$ is non-triangular, it follows
            that $u$ is not adjacent to any of the neighbours of $w$. As explained in
            \Cref{claim:C-indep-dense-case}, $w$ is adjacent to all but at most
            $O(\myeta n)$ vertices in $B$.  Therefore, $u$ had at most $O(\myeta
            n)$ neighbours in $B$ and, likewise, so does $v$. 
 
            Let $G'$ be the graph obtained by replacing $u$ and $v$ with new vertices
            $x$ and $y$ where $x$ is adjacent to all of $A \cup C$ and $y$ is
            adjacent to all of $(A \cup B) \setminus \{u, v\}$. We have $\t(G')
            \ge \t(G)$ and $e(G') \ge e(G) + |B| - 2 - O(\myeta n) > e(G)$,
            contradicting the optimality of $G$. Therefore, there are no
            $B$ is independent.
        \end{proof}

        We have proved that $B$ and $C$ are independent, $A$ is complete,
        and its vertices are adjacent to all of $B$ and none of $C$. We may add
        any missing edges between $B$ and $C$ without creating new
        triangles, so by the optimality of $G$, the vertices in $B$ are
        connected to all of $C$.  It follows that $G$ is isomorphic to $G(|A|,
        |B|, |C|)$, completing the proof of \Cref{thm:almost-complete}.
    \end{proof}

\section{Concluding remarks} \label{sec:conclusion}

    We note that we have not fully resolved
    \Cref{conj:chracterise-extremal-examples}.    
    \conjMain*
    
    \Cref{thm:main} shows that the minimum number of non-triangular edges among
    $n$-vertex graphs with $e$ is attained on a graph $G(a, b, c)$. However, we
    have not shown that such graphs are the only minimisers. 
    Nevertheless, we believe that this fact can be proved (for sufficiently
    large $n$) by retracing our proofs. Since our paper is already quite
    long, we spare the reader any further details.  In any case, we are only
    able to prove the conjecture for sufficiently large $n$, and it would be
    interesting to extend our result to work for all $n$.
    
    We have not specified explicitly how large $n$ should in order for our proof
    to hold, mainly because, due to the complexity of the proof, it is quite
    hard to find such an explicit bound. Nevertheless, we expect this bound to be
    `reasonably small' (say, much smaller than a bound that may arise from the
    use of the regularity lemma), because the inequalities we need to hold are
    polynomial in $n$.

    The following question arises from
    \Cref{conj:chracterise-extremal-examples}, by considering edges on $K_r$ for
    $r \ge 4$.

    \begin{prob}
        How many edges in copies of $K_r$ must an $n$-vertex
        graph with $e$ edges have? Or, more generally, which $n$-vertex graph
        with $e$ edges minimise the number of edges contained in $K_r$'s?
    \end{prob}
    
    It seems reasonable to believe that the
    extremal examples are analogues of graphs $G(a, b, c)$, namely, they may be
    formed by adding a clique to one of the parts of a complete $(r-1)$-partite
    graph with $n$ vertices. We believe that the methods used in this
    paper may be useful when tackling this more general problem.

    We mention another possible generalisation, where instead of minimising the
    number of triangular edges, one wishes to minimise the number of edges
    contained in copies of an odd cycle.

    \begin{prob} \label{prob:odd-cycle}
        How many edges in copies of $C_{2k + 1}$ must an $n$-vertex graph with
        $e$ edges have? Which graphs minimise this quantity?
    \end{prob}

    It turns out that the case $k \ge 2$ is quite different from $k = 1$
    (i.e.~a triangle). Erd\H{o}s, Faudree and Rousseau
    \cite{erdos-faudree-rousseau} proved that for any fixed $k \ge 2$, any graph
    with $n$ vertices and $\floor{n^2 / 4} + 1$ edges has at least $\frac{11}{144}n^2
    - O(n)$ edges in copies of $C_{2k + 1}$, whereas the number of edges on
    triangles is, as mentioned in the introduction, at least $2\floor{n / 2} +
    1$, and the latter bound is best possible. So, the jump in the number of
    $C_{2k + 1}$-edges (for $k \ge 2$) is very sharp, while the jump in the number of
    triangular edges is much smoother.
    
    In the same paper, Erd\H{o}s, Faurdree and Rousseau conjectured a more
    precise statement: for any fixed $k = 2$, there are at most $n^2/36 + O(n)$
    non-pentagonal edges (where a \emph{pentagonal} edge is an edge contained in
    a $C_5$). This can be attained by the union of a complete graph on roughly
    $n/3$ vertices and a complete bipartite graph on the remaining vertices.
    However, an example by F\"uredi and Maleki (see
    the last section in \cite{furedi-maleki}) shows that the conjecture is false: 
    there are $n$-vertex graphs with $\floor{n/4} + 1$ edges and
    $n^2/27.3\ldots + O(n)$ non-pentagonal edges. The example is somewhat
    similar to a graph $G(a, b, c)$: here we have four sets, $A, B, C, D$ of
    sizes $a, b, c, d$, such that $A$ induces a clique, and all possible $A-B$, $B-C$ and
    $C-D$ edges are present. The
    non-pentagonal edges are the $C-D$ edges, and the above value is obtained by
    optimising $a, b, c, d$.
    They are able to calculate, asymptotically, the minimum number of edges on
    $C_{2k + 1}$ (where $k \ge 2$) among $n$-vertex graphs with $e$ edges, where
    $e = \gamma n^2$ for any fixed $1/4 < \gamma < 1/2$. In particular, it turns
    out that the
    example of Erd\H{o}s, Faudree and Rousseau \cite{erdos-faudree-rousseau} is
    asymptotically best possibly for $C_{2k + 1}$ when $k \ge 3$.
    It would be interesting to prove exact versions of these results, in
    particular, it would be interesting to resolve \Cref{prob:odd-cycle} for $e
    = \floor{n^2/4} + 1$, where the methods of F\"uredi and Maleki do not apply.

    Finally, we mention that all the aforementioned problems are special cases
    of the following general problem.
    
    \begin{prob}

        What is the minimum number of edges contained in copies of $F$ among
        $n$-vertex graphs with $e$ edges (where $F$ is any fixed graph)?
        Moreover, what are the extremal examples?

    \end{prob}
    
    F\"uredi and Maleki \cite{furedi-maleki-pentagon} calculated the minimum,
    asymptotically, for $3$-chromatic graphs $F$. For any other $F$,
    this problems is wide open. Finally, we note that it is possible to go even
    further and generalise the problem to the context of hypergraphs. 

    \subsection*{Acknowledgements} 
    
        We would like to thank B\'ela Bollob\'as for bringing this problem to
        our attention. The work on this project started during our stay
        in IMT Lucca, and we would like to thank the institute for their
        hospitality.
    
    \bibliography{trianglebib}
    \bibliographystyle{amsplain}

\end{document}